\documentclass[12pt]{amsart}

\usepackage{amssymb,amsbsy}
\usepackage{latexsym,color,longtable}
\usepackage{verbatim,euscript}
\usepackage{fullpage,array}
\usepackage{tikz}
\usetikzlibrary{arrows,shapes,trees}
\usepackage[matrix,arrow,curve]{xy}

\numberwithin{equation}{section}
\usepackage[colorlinks=true,linkcolor=blue,urlcolor=violet,citecolor=magenta]{hyperref}

\input {cyracc.def}
\font\tencyr=wncyr10 
\font\tencyi=wncyi10 
\font\tencysc=wncysc10 
\def\rus{\tencyr\cyracc}
\def\rusi{\tencyi\cyracc}
\def\rusc{\tencysc\cyracc}

\catcode`\@=\active
\catcode`\@=11

\renewcommand{\@cite}[2]{[{{\bf #1}\if@tempswa , #2\fi}]}
\renewcommand{\@biblabel}[1]{[{\bf #1}]\hfill}
\catcode`\@=12

\newtheorem{thm}{Theorem}[section]
\newtheorem{lm}[thm]{Lemma}
\newtheorem{cl}[thm]{Corollary}
\newtheorem{prop}[thm]{Proposition}

\theoremstyle{remark}
\newtheorem{rmk}[thm]{Remark}

\theoremstyle{definition}
\newtheorem{ex}[thm]{Example}

\newenvironment{E6}[6]{%
{\small\begin{tabular}{@{}c@{}}
{#1}--{#2}--\lower3.5ex\vbox{\hbox{{#3}\rule{0ex}{2.5ex}}
\hbox{\hspace{0.4ex}\rule{.1ex}{1ex}\rule{0ex}{1.4ex}}\hbox{{#6}\strut}}--{#4}--{#5}
\end{tabular}}}

\newcommand {\be}{{\mathfrak b}}
\newcommand {\ce}{{\mathfrak c}}

\newcommand {\g}{{\mathfrak g}}
\newcommand {\h}{{\mathfrak h}}
\newcommand {\ka}{{\mathfrak k}}
\newcommand {\el}{{\mathfrak l}}
\newcommand {\me}{{\mathfrak m}}

\newcommand {\es}{{\mathfrak s}}
\newcommand {\te}{{\mathfrak t}}
\newcommand {\ut}{{\mathfrak u}}
\newcommand {\z}{{\mathfrak z}}

\newcommand {\gln}{{\mathfrak {gl}}_n}
\newcommand {\sln}{{\mathfrak {sl}}_n}
\newcommand {\slno}{{\mathfrak {sl}}_{n+1}}
\newcommand {\sltn}{{\mathfrak {sl}}_{2n}}

\newcommand {\spv}{{\mathfrak {sp}}(\VV)}
\newcommand {\spn}{{\mathfrak {sp}}_{2n}}

\newcommand {\son}{{\mathfrak {so}}_{n}}

\newcommand {\sfr}{\eus R}
\newcommand {\VV}{\eus V}
\newcommand {\vlb}{\eus V_\lb}
\newcommand {\eus}{\EuScript}

\newcommand {\gP}{{\eus P}}

\newcommand {\ap}{\alpha}
\newcommand {\Lb}{\varLambda}
\newcommand {\lb}{\lambda}
\newcommand {\vp}{\varphi}
\newcommand {\vth}{\vartheta}
\newcommand {\blb}{\boldsymbol{\lambda}}

\newcommand {\N}{{\mathfrak N}}
\newcommand {\co}{{\mathcal O}}

\newcommand {\BC}{{\mathbb C}}

\newcommand {\BN}{{\mathbb N}}
\newcommand {\BQ}{{\mathbb Q}}
\newcommand {\BZ}{{\mathbb Z}}
\newcommand {\BP}{{\mathbb P}}

\newcommand {\md}{/\!\!/}

\newcommand {\ad}{{\mathrm{ad\,}}}

\newcommand {\codim}{{\mathrm{codim\,}}}
\newcommand {\cha}{{\mathsf{char}}}

\newcommand {\hot}{{\mathsf{ht}}}

\newcommand {\ind}{{\mathsf{ind\,}}}
\newcommand {\Lie}{{\mathrm{Lie\,}}}

\newcommand {\Ima}{{\mathsf{Im}}}

\newcommand {\rk}{{\mathsf{rk}}}
\newcommand {\spe}{{\mathsf{Spec\,}}}

\newcommand {\tri}{{\mathfrak{sl}}_2}

\newcommand {\GR}[2]{{\textrm{{\sf\bfseries #1}}}_{#2}}

\newcommand {\ov}{\overline}
\newcommand {\un}{\underline}

\newcommand {\beq}{\begin{equation}}
\newcommand {\eeq}{\end{equation}}

\renewcommand{\le}{\leqslant}
\renewcommand{\ge}{\geqslant}
\renewcommand{\lg}{\langle}
\newcommand{\rg}{\rangle}
\newcommand{\bbk}{\Bbbk}

\newcommand {\omin}{\co_{\sf min}}
\newcommand {\bomin}{\ov{\co_{\sf min}}}
\newcommand {\sat}{\mathsf{Sat}}

\newcommand {\Sec}{\textrm{{\sl\bfseries Sec}}}
\newcommand {\CS}{\textrm{{\sl\bfseries CS}}}
\newcommand {\Nsp}{\N^{\sf sph}}
\newcommand {\psim}{\psi_{\sf sph}}
\newcommand {\psimo}{\psi_{\co}}

\begin{document}
\setlength{\parskip}{2pt plus 4pt minus 0pt}
\hfill {\scriptsize May 6, 2021} 
\vskip1ex

\title[Projections of $\co_{\sf min}$]{Projections of the minimal nilpotent orbit in a simple Lie algebra and 
secant varieties}
\author{Dmitri I. Panyushev}
\address{
Institute for Information Transmission Problems of the R.A.S., 
\hfil\break\indent  
Bolshoi Karetnyi per.~19, Moscow 127051, Russia}
\email{panyushev@iitp.ru}
\thanks{This research was funded by RFBR, project {\rus N0} 20-01-00515.}
\begin{abstract}
Let $G$ be a simple algebraic group with $\g=\Lie G$ and $\omin\subset\g$ the minimal nilpotent orbit.
For a $\BZ_2$-grading $\g=\g_0\oplus\g_1$, let $G_0$ be a connected subgroup of $G$ with 
$\Lie G_0=\g_0$. We study the $G_0$-equivariant projections $\vp:\bomin\to \g_0$ and
$\psi:\bomin\to\g_1$. It is shown that the properties of $\ov{\vp(\omin)}$  and $\ov{\psi(\omin)}$ 
essentially depend on whether the intersection $\omin\cap\g_1$ is empty or not. If 
$\omin\cap\g_1\ne\varnothing$, then both $\ov{\vp(\omin)}$  and $\ov{\psi(\omin)}$ contain a 
$1$-parameter family of closed $G_0$-orbits, while if $\omin\cap\g_1=\varnothing$, then both are
$G_0$-prehomogeneous. We prove that $\ov{G{\cdot}\vp(\omin)}=\ov{G{\cdot}\psi(\omin)}$. Moreover,
if $\omin\cap\g_1\ne\varnothing$, then this common variety is the affine cone over the secant 
variety of $\BP(\omin)\subset\BP(\g)$. 
As a digression, we obtain some invariant-theoretic results on the affine cone over the secant variety 
of the minimal orbit in an arbitrary simple $G$-module. In conclusion, we discuss more general 
projections that are related to either arbitrary reductive 
subalgebras of $\g$ in place of $\g_0$ or spherical nilpotent $G$-orbits in place of $\omin$.
\end{abstract}
\maketitle

\tableofcontents
\section{Introduction}

\noindent
Let $G$ be a simple algebraic group with $\Lie G=\g$, $\N$ the nilpotent cone in $\g$, and 
$\omin\subset \N$ the minimal non-trivial nilpotent $G$-orbit. The ground field $\bbk$ is algebraically 
closed and $\cha(\bbk)=0$. Let $\sigma$ be an involution of $\g$ and $\g=\g_0\oplus\g_1$ the 
corresponding $\BZ_2$-grading. Write $G_0$ for the connected (reductive) subgroup of $G$ with 
$\Lie G_0=\g_0$. In this article, we study invariant-theoretic properties of the $G_0$-equivariant 
projections $\vp:\bomin\to \g_0$ and $\psi:  \bomin\to \g_1$. The initial motivation came from the 
observation that if $\omin\cap\g_0=\varnothing$, then $\psi$ is a finite morphism (and likewise for $\vp$,
if $\omin\cap\g_1=\varnothing$). Our general description shows that the structure of both $\Ima(\vp)$ and 
$\Ima(\psi)$ crucially depends on the fact whether $\omin\cap\g_1$ is empty or not. 

By the Morozov--Jacobson theorem, any nonzero $e\in\N$ can be included in an $\tri$-triple $\{e,h,f\}$. 
Here $h$ is a semisimple element of $\g$, which is called a {\it characteristic} of $e$. If $e,f \in\g_1$ and
$h\in\g_0$, then such a triple is said to be {\it normal}. By~\cite{kr71}, a normal 
$\tri$-triple exists for any $e\in\N \cap\g_1$. Another reason for considering the intersection
$G{\cdot}e\cap\g_1$ is that 
\[
      G{\cdot}e\cap\g_1\ne \varnothing \ \Longleftrightarrow \ G{\cdot}h\cap\g_1\ne \varnothing  ,
\]
see~\cite{leva}.  If $e\in\omin$, then the corresponding $\tri$-triples are said to be {\it minimal}.
We begin our study of $\vp$ and $\psi$ with classifying the $\BZ_2$-gradings (involutions) of $\g$ such 
that $\omin\cap\g_0=\varnothing$ or $\omin\cap\g_1=\varnothing$. (It appears that there are two cases
for the former and six cases for the latter.) The equality $\omin\cap\g_1=\varnothing$ is equivalent to 
that $G_0$ has a dense orbit in $\omin$ and some other interesting properties, see 
Theorem~\ref{thm:g1-empty}.

For a subvariety $X$ of a vector space $\VV$, we set 
\[
   \eus K(X):=\ov{\bigcup_{t\in\Bbbk^*}tX}=\ov{\bbk^*{\cdot}X} .
\] 
It is a closed cone in $\VV$, and $\dim\eus K(X)$ is either $\dim X$ or $\dim X+1$. We say that $X$ is 
{\it conical}, if $\dim\eus K(X)=\dim X$. 
Our description of $\ov{\vp(\omin)}$ and $\ov{\psi(\omin)}$ can briefly be summarised in the following
two theorems, see Section~\ref{sect:projections} for the complete account.

\begin{thm}      \label{thm:intro-non-empty}
Suppose that $\omin\cap\g_1\ne \varnothing$, and let $\{e,h,f\}$ be a normal minimal $\tri$-triple. Then
\\[.6ex]
\centerline{$\ov{\vp(\omin)}=\eus K(G_0{\cdot}h)$ \ and \  $\ov{\psi(\omin)}=\eus K(G_0{\cdot}h_1)$,} 
\\[.6ex]
where $h_1=e-f$. Here\/ $\dim\vp(\omin)=\dim\omin-1$ and $\dim\psi(\omin)=\dim\omin$. Furthermore,
$\deg(\psi)=2$ and we also describe generic fibres of $\vp$ and $\psi$.
\end{thm}

Here both varieties contain a $1$-parameter family of closed $G_0$-orbits as a dense subset. 
Actually,  $h_1$ can be replaced with an arbitrary non-trivial linear combination of $e$ and $f$.

\begin{thm}    \label{thm:intro-empty}
Suppose that $\omin\cap\g_1 = \varnothing$. Then both $\ov{\vp(\omin)}$ and
$\ov{\psi(\omin)}$ contain a dense conical $G_0$-orbit. In this case,
$\dim\vp(\omin)=\dim\omin$ \ and \ $\dim\psi(\omin)<\dim\omin$.
\end{thm}

An important structure that naturally comes into play is the {\it secant variety\/} of the projective variety 
$\BP(\omin)=G{\cdot}[e]\subset\BP(\g)$. Write $\Sec(\BP\omin)$ or $\Sec(G{\cdot}[e])$ for this secant 
variety, and let $\CS(\omin)\subset\g$ denote the affine cone over it. A simple observation that relates the 
secant variety to the morphisms $\vp$ and $\psi$ is that
\[
    \vp(\omin) \subset \CS(\omin)\cap\g_0 \ \ \text{ and } \ \ \psi(\omin) \subset \CS(\omin)\cap\g_1 ,
\]
see Lemma~\ref{lm:secant-sigma}. The proofs of Theorems~\ref{thm:intro-non-empty} and
\ref{thm:intro-empty} exploit the equality 
$\Sec(G{\cdot}[e])=\ov{G{\cdot}[h]}$~\cite{koy99}, where $h$ is a characteristic of $e\in\omin$. In our 
notation, this is equivalent to that $\CS(\omin)=\eus K(G{\cdot}h)$. 

As a digression from the main route, a description of the secant variety for the 
minimal $G$-orbit $\omin(\VV)$ in a simple $G$-module $\VV$ and its secant defect is given
(Theorem~\ref{thm:sec-min-orb}). We also derive some
invariant-theoretic consequences of it. One of the applications is the list of all simple $G$-modules 
$\VV$, where $G$ is simple, such that $\CS(\omin(\VV))=\VV$, see Section~\ref{subs:IT-sledstviya}. 

\noindent
Let us provide a selection of other results.
\\ \indent \textbullet \ \ 
if $\g_1$ is a simple $\g_0$-module, so that the minimal $G_0$-orbit $\omin(\g_1)$ exists, then
some properties of $\CS(\omin(\g_1))$ are given, see Section~\ref{subs:omin-g1}.
\\ \indent \textbullet \ \ Our theory implies that $\ov{G{\cdot}\vp(\omin)}=\ov{G{\cdot}\psi(\omin)}$ for all $\BZ_2$-gradings, see Section~\ref{subs:rel-between}.
\\ \indent \textbullet \ \ 
We study in details the cases in which $\vp$ or $\psi$ is a finite morphism. If 
$\omin\cap\g_0=\varnothing$, then we prove that $\Ima(\psi)$ is normal and $\deg(\psi)=2$. Furthermore,
the map $\bomin\to \psi(\bomin)$ is the categorical quotient {w.r.t.} the linear action of $\BZ/2\BZ$. 
Here the generator of $\BZ/2\BZ$ is $(-\sigma)\in GL(\g)$. If $\omin\cap\g_1=\varnothing$, then the 
almost same results are true for $\vp$. The only difference is that the generator of $\BZ/2\BZ$ is 
$\sigma$, see Propositions~\ref{prop:faktor-C2-psi} and \ref{prop:faktor-C2-fi}.

In Section~\ref{sect:more}, we briefly discuss possible generalisations of our main setting related to
involutions and $\omin\subset\N$. For instance, let $H\subset G$ be semisimple and 
$\g=\h\oplus\me$, where $\me=\h^\perp$. We note that if $\omin\cap \me=\varnothing$, then the 
projection (with kernel $\me$) of $\bomin$ to $\h$ is finite and has other interesting properties. 
Another possibility is to consider projections of spherical nilpotent orbits for some special involutions.

\un{Main notation}. 
Let $G$ be a semisimple algebraic group and $\g=\Lie G$. Then
\begin{itemize}
\item[--] $G^x$ is the isotropy group of $x\in\g$ in $G$ and $\g^x=\Lie G^x$ is the centraliser of $x$ in $\g$;
\item[--] $\Phi$ is the Killing form on $\g$ and $\h^\perp$ is the orthocomplement of $\h\subset\g$ {w.r.t.}
$\Phi$;
\item[--] $\te$ is a Cartan subalgebra of $\g$ and $\Delta=\Delta(\g,\te)$ is the root system with respect to 
$\te$;
\item[--] $\te^*_\BQ$ is the $\BQ$-vector subspace of $\te^*$ generated by $\Delta$ and $(\ ,\ )$ is the positive-definite form on $\te^*_\BQ$ induced by $\Phi$; as usual, $\gamma^\vee=2\gamma/(\gamma,\gamma)$ for $\gamma\in\Delta$.
\item[--] Given $\gamma\in\Delta$, then $\g_\gamma$ is the root space in $\g$ and 
$e_\gamma\in\g_\gamma$ is a nonzero root vector;
\item[--] If $\Delta^+\subset\Delta$ is a set of positive roots, then $\Pi\subset\Delta^+$ is a set of simple roots;
\item[--] If $\co=G{\cdot}e\subset\N$ is a nontrivial orbit, then $\eus D(e)$ or $\eus D(\co)$ is its weighted 
Dynkin diagram;
\end{itemize}

A direct sum of Lie algebras is denoted by `$\dotplus$' and $\mathsf{Inv}(\g)$ is the set of involutions of 
$\g$. 
\\ \indent
Our man reference for algebraic groups and Lie algebras is~\cite{t41}.

\section{Recollections on nilpotent orbits, involutions, and Satake diagrams} 
\label{sect:2}

\noindent
Let $\g$ be a semisimple Lie algebra. 
For $e\in\N\setminus\{0\}$, let $\{e,h,f\}$ be 
an $\tri$-triple in $\g$, i.e., $[h,e]=2e$, $[e,f]=h$, and $[h,f]=-2f$~\cite[Ch.\,6,\,\S2]{t41}. The semisimple 
element $h$ is called a {\it characteristic\/} of $e$. It is important that all $\tri$-triples with a given $h$ are 
$G^h$-conjugate ({Mal'cev}) and all $\tri$-triples with a given $e$ are $G^e$-conjugate (Kostant)~\cite[3.4]{CM}.
Without loss of generality, one may assume that $h\in\te$ and $\ap(h)\ge 0$ for all $\ap\in\Pi$. Then $h$ 
is said to be the {\it dominant characteristic} (w.r.t. chosen $\te$ and $\Delta^+$). By a celebrated result 
of {Dynkin} (1952), one then has $\ap(h)\in\{0,1,2\}$~\cite[Ch.\,6,\,\S2,\,Prop\,2.2]{t41}. 
The {\it weighted Dynkin diagram\/} of $\co=G{\cdot}e\subset\N$, 
$\eus D(\co)$, is the Dynkin diagram of $\g$ equipped with labels $\{\ap(h)\}_{\ap\in\Pi}$. The 
{\it set of zeros\/} of $\eus D(\co)$ is $\Pi_{0,\co}=\{\ap\in\Pi\mid \ap(h)=0\}$.

Let $\g=\bigoplus_{i\in\BZ}\g(i)$ be the $\BZ$-grading determined by $h$, i.e.,  
$\g(j)=\{v\in \g \mid [h,v]=jv\}$. It will be referred to as the $(\BZ,h)$-{\it grading\/} of $\g$.
Here $\g(0)=\g^h$ and $e\in\g(2)$. 
A nonzero $e\in\N$ is said to be {\it even}, if the $h$-eigenvalues in $\g$ are even.
\\ \indent
Set $\g({\ge}j)=\bigoplus_{i\ge j}\g(i)$.  It follows from the $\tri$-theory that 
$\ad e:\g(i)\to \g(i+2)$ is injective (resp. surjective) if $i\le -1$ (resp. $i\ge -1$). Hence
 $\g^e\subset \g({\ge} 0)$, $\dim\g^e=\dim\g(0)+\dim\g(1)=\dim\g^h+\dim\g(1)$, and $e$ is even if and only if $\dim\g^e=\dim\g^h$. Furthermore, $(\ad e)^j: \g(-j) \to \g(j)$ is bijective.
 The {\it height}\/ of $\co$, denoted $\hot(\co)$, is $\max\{j\mid \g(j)\ne 0\}$. Equivalently,
$\hot(\co)=\max\{j\mid (\ad e)^j\ne 0\}$. The height of $\co$ can be determined via $\eus D(\co)$, and
if $\g$ is classical, then it is easily computed via the partition of $\co$, see~\cite[Section\,2]{aif99}.

\noindent  
Let $\g=\g_0\oplus\g_1$ be a $\BZ_2$-grading
and $\sigma$ the related involution of $\g$, i.e., $\g_i$ is the $(-1)^i$-eigenspace of $\sigma$.
Then $\g_0$ is reductive, but not necessarily semisimple, and $\g_1$ is an orthogonal $\g_0$-module. 
We also say that  $(\g,\g_0)$ is a {\it symmetric pair}. Whenever we wish to stress that $\g_i$ is defined via certain $\sigma\in\mathsf{Inv}(\g)$, 
we write $\g_i^{(\sigma)}$ for it. We can also write $\g^\sigma$ in place of $\g^{(\sigma)}_0$.

If $\g$ is simple, then either $\g_0$ is semisimple and $\g_1$ is a simple $\g_0$-module, or $\g_0$ has 
a $1$-dimensional centre and $\g_1$ is a sum of two simple $\g_0$-modules, which are dual to each 
other~\cite[Ch.\,4,\,\S 1.4]{t41}. The second case is always related to a {\it short $\BZ$-grading\/}
of $\g$. That is, there is a $\BZ$-grading $\g=\g(-1)\oplus\g(0)\oplus\g(1)$ such that
$\g_0=\g(0)$ and $\g_1=\g(-1)\oplus\g(1)$.  
Here $\g(0)\oplus\g(1)$ is a maximal parabolic subalgebra of $\g$ with abelian nilradical.

The centraliser  $\g^x$ is $\sigma$-stable for any $x\in \g_i$, hence $\g^x=\g^x_0\oplus\g^x_1$. 
The following result of {Kostant--Rallis} is frequently used below.
\begin{prop}[cf.~{\cite[Proposition\,5]{kr71}}]             \label{prop:kr71}
         If\/ $x\in\g_1$, then $\dim G_0{\cdot}x=\frac{1}{2}\dim G{\cdot}x$.
\end{prop}
A {\it Cartan subspace\/} of $\g_1$ is a maximal abelian subspace consisting of semisimple 
elements. By~\cite{kr71}, all Cartan subspaces of $\g_1$ are $G_0$-conjugate and if  $\ce$ is a
Cartan subspace, then $G_0{\cdot}\ce$ is dense in $\g_1$. If $x\in\ce$ is generic, then $\g_1^x=\ce$.

The {\it rank\/} of the symmetric variety $G/G_0$, $r(G/G_0)$, can be defined in many equivalent ways.
We use the following
\beq     \label{eq:rang-symm-var}
   r(G/G_0)=\dim\ce =\dim\g_1-\max_{x\in\g_1}\dim G_0{\cdot}x .
\eeq
\subsection{Satake diagrams}   
\label{subs:Satake}
It is known that (for $\bbk=\BC$) the real forms of $\g$ are represented by the {\it Satake diagrams} (see 
e.g.~\cite[Ch.\,4,\S\,4.3]{t41}) and there is a one-to-one correspondence between the real forms and 
$\BZ_2$-gradings of $\g$. Thereby, one associates the Satake diagram to an involution (symmetric 
pair), cf.~\cite{spr87}. The Satake diagram of $\sigma\in\mathsf{Inv}(\g)$, denoted 
$\mathsf{Sat}(\sigma)$, is the Dynkin diagram of $\g$, with black 
and white nodes, where certain pairs of white nodes can be joined by an arrow. 
Following~\cite[Sect.\,1]{CP83}, we give an approach to constructing $\sat(\sigma)$ that does not refer to
$\BC$ and real forms.
Let $\te=\te_0\oplus\te_1$ be a $\sigma$-stable Cartan subalgebra of
$\g$ such that $\dim\te_1$ is maximal, i.e., $\te_1$ is a Cartan subspace of $\g_1$.
Let $\Delta$ be the root system of $(\g,\te)$. Since $\te$ is $\sigma$-stable, $\sigma$ acts on $\Delta$.
One can choose the set of positive roots, $\Delta^+$, such that if $\gamma\in\Delta^+$ and 
$\gamma\vert_{\te_1}\ne 0$, then $\sigma(\gamma)\in -\Delta^+$. Let $\Pi$ be the set of simple roots 
in $\Delta^+$ and $\{\varpi_\ap\mid \ap\in\Pi\}$ the corresponding fundamental weights. We identify  
$\Pi$ with the nodes of the Dynkin diagram.
Set $\Pi_0=\Pi_{0,\sigma}=\{\ap\in \Pi\mid \ap\vert_{\te_1}\equiv 0\}$ and $\Pi_1=\Pi_{1,\sigma}=
\Pi\setminus\Pi_0$. Then
\begin{itemize}
\item the node $\ap$ is black if and only if $\ap\in\Pi_0$;
\item if $\ap\in\Pi_0$, then $\sigma(\ap)=\ap$ and the root space $\g_\ap$ is contained in $\g_0$; 
\item if $\ap\in\Pi_1$, then $\sigma(\ap)=-\beta+\sum_j a_j\nu_j$
for some $\beta\in\Pi_1$ and $\nu_j\in\Pi_0$. Then also $\sigma(\varpi_\ap)=-\varpi_\beta$.
In this case, if 
$\ap\ne\beta$, then the white nodes $\ap$ and $\beta$ are joined by an arrow. If $\ap=\beta$, then there is no arrow attached to $\ap$.
\end{itemize}
If $\sat(\sigma)$ has  $k$ arrows, i.e., $\Pi_1$ has $k$ pairs of simple roots 
$(\ap_i,\beta_i)$ such that $\sigma(\varpi_{\ap_i})=-\varpi_{\beta_i}$, then $\dim\te_0=\# \Pi_0+k$ and 
$\dim\te_1=\#\Pi_1-k$. More precisely, let 
$\{h_\ap\mid \ap\in\Pi\}\subset \te$ be Chevalley generators, i.e., $[h_\ap,e_\gamma]=(\ap^\vee,\gamma)e_\gamma$ for $\gamma\in\Delta$.
Then
\[
   \te_0=\lg h_\ap\mid \ap\in\Pi_0\rg \oplus \lg h_{\ap_i}-h_{\beta_i} \mid i=1,\dots,k\rg .
\] 
If $\Pi_1=\{\ap_1,\beta_1,\dots,\ap_k,\beta_k,\nu_1,\dots,\nu_s\}$ and we identify $\te_1$ and $\te_1^*$ via the Killing form, then
$\te_1=\lg \varpi_{\ap_i}+\varpi_{\beta_i}\mid i=1,\dots,k\rg\oplus 
\lg \varpi_{\nu_j}\mid j=1,\dots,s\rg$. Moreover, if $\be$ is the Borel subalgebra corresponding to $\Delta^+$, then $\be+\g_0=\g$ and $\be\cap\g_1=\te_1$.

Let $x\in\te_1$ be a generic semisimple element. Then $\g^x_1=\te_1$ and $\mathsf{Sat}(\sigma)$ 
encodes the structure of $\g^x_0$. Namely, $\g^x$ is a Levi subalgebra such that the Dynkin diagram of 
$[\g^x,\g^x]=[\g^x_0,\g^x_0]$ is the subdiagram $\Pi_0$.
Here the number of arrows equals $\dim(\g^x_0/[\g^x_0,\g^x_0])$. 

Recall that an algebraic subgroup $H\subset G$ is said to be {\it spherical}, if $B$ has an open orbit in 
$G/H$. Then one also says that $G/H$ is a {\it spherical homogeneous space} and $\h=\Lie H$ is a 
{\it spherical subalgebra} of $\g$. The above relation $\be+\g_0=\g$ exhibits the well-known fact that 
$G_0\subset G$ is spherical for any $\sigma$.
\begin{ex}   
\label{ex:max-rang}
As $G/G_0$ is a spherical homogeneous space,  $\dim\g_1=\dim (G/G_0)\le \dim B$. Hence
$\dim\g_0\ge \dim U$ and  $\dim\g_1-\dim\g_0\le\rk(\g)$ for any $\sigma\in\mathsf{Inv}(\g)$. If 
$\dim\g_1-\dim\g_0=\rk(\g)$, then $r(G/G_0)=\rk(\g)$ and $\sigma$ is said to be of {\it maximal rank}. (In~\cite{spr87}, such involutions are called {\it split}.)
Equivalently, $\g_1$ contains a Cartan subalgebra of $\g$. For any simple $\g$, there is a unique, up to 
$G$-conjugacy, involution of maximal rank, and we denote it by $\vartheta_{\sf max}$. In this case, 
$\g^x\cap \g_0^{(\vartheta_{\sf max})} =\{0\}$ for a generic $x\in\g_1^{(\vartheta_{\sf max})}$. Hence 
$\mathsf{Sat}(\vartheta_{\sf max})$ has neither black nodes nor arrows. Yet another characterisation is 
that $\vartheta_{\sf max}$ corresponds to a split real form of $\g$, see~\cite[Ch.\,4,\,\S\,4.4]{t41}.
\end{ex}

\begin{rmk}    \label{rem:Levon}
By a fundamental result of {Antonyan}, for any $\sigma\in\mathsf{Inv}(\g)$ and any $\tri$-triple 
$\{e,h,f\}\subset\g$, one has $G{\cdot}e\cap\g_1\ne\varnothing$ if and only if 
$G{\cdot}h\cap\g_1\ne\varnothing$, see~\cite[Theorem\,1]{leva}. This readily implies that 
$G{\cdot}x\cap\g_1\ne\varnothing$ for {\bf any} $x\in\g$ if and only if $\sigma=\vartheta_{\sf max}$  \ (\cite[Theorem\,2]{leva}). For arbitrary $\sigma$ and $e\in\N$, Antonyan's results imply that 
$G{\cdot}e\cap\g_1\ne\varnothing$ {\sl if and only if} \\ \indent
{\bf --}  \ the black nodes of $\mathsf{Sat}(\sigma)$ 
are contained in the set of zeros of $\eus D(e)$, i.e., $\Pi_{0,\sigma}\subset \Pi_{0,G{\cdot}e}$; 
\\  \indent {\bf --} \
$\ap(h)=\beta(h)$ whenever 
$\ap,\beta\in\Pi_{1,\sigma}$ are joined by an arrow in $\mathsf{Sat}(\sigma)$.
\end{rmk}

\section{Minimal orbits and their secant varieties} 
\label{sect:HV}

\noindent
Let $\vlb$ be a simple $G$-module with highest weight $\lb$ (with respect to some choice of $T\subset B$). Then $\omin(\vlb)$ denotes the $G$-orbit of the highest weight 
vectors. It is the unique nonzero $G$-orbit of minimal dimension whose closure contains the origin.
If $v_\lb\in\vlb$ is a highest weight vector (i.e., a $B$-eigenvector), then 
$\omin(\vlb)=G{\cdot}v_\lb$.
The orbit $\omin(\vlb)$ is stable under homotheties, i.e., the dilation action of $\bbk^*$ on $\vlb$, 
and $\BP(\omin(\vlb))$ is the unique closed $G$-orbit in $\BP(\vlb)$. Hence
\beq   \label{eq:closure-omin}
     \ov{\omin(\vlb)}=\omin(\vlb)\cup\{0\} ,
\eeq 
see~\cite[Theorem~1]{vp72}. We say that $\omin(\vlb)$ is the {\it minimal ($G$-)orbit}\/ in  $\vlb$.

Let us recall some basic properties of the minimal orbits established in~\cite[\S\,1]{vp72}:
\begin{itemize}
\item for any $\lb$, the variety $\ov{\omin(\vlb)}$ is normal;
\item the algebra of regular functions $\bbk[\ov{\omin(\vlb)}]$ is $\BN$-graded and the component of 
grade $n$ is the $G$-module $(\VV_{n\lb})^*$, the dual of $\VV_{n\lb}$;
\item $\ov{\omin(\vlb)}$ is a factorial variety (i.e., $\bbk[\ov{\omin(\vlb)}]$ is a UFD) if and only if $\lb$ 
is fundamental.
\end{itemize}

\begin{lm}        \label{lm:omin-plane}
Let $\gP\subset\vlb$ be a $2$-dimensional subspace. Suppose that
$\gP\cap\ov{\omin(\vlb)}$ contains at least three different lines. Then $\gP\subset
\ov{\omin(\vlb)}$.
\end{lm}
\begin{proof}
This readily follows from the fact that the ideal of  $\ov{\omin(\vlb)}$ in $\bbk[\vlb]$ 
is generated by quadrics, see e.g.~\cite{L}.  
\end{proof}

\subsection{Secant varieties}    
\label{subs:secant}
Let $\Sec(X)$ denote the {\it secant variety\/} of an irreducible projective variety 
$X\subset \BP^N=\BP(V)$. By definition, $\Sec(X)$ is the closure of the union of all secant lines to $X$. 
If $x,y\in X$ are different points and $\BP^1_{xy}\subset \BP(V)$ is the line through $x$ and $y$, then
\[
    \Sec(X)=\ov{\bigcup_{x,y\in X} \BP^1_{x,y}} ,
\]
see e.g.~\cite[Chapter\,5]{land}. It follows that $\Sec(X)$ contains all tangent lines to $X$, too. If $X$ is irreducible, then so is $\Sec(X)$, and the expected dimension of $\Sec(X)$ is 
$\min\{2\dim X+1, N\}$. If $\dim\Sec(X) < \min\{2\dim X+1, N\}$,  
then $\Sec(X)$ is said to be {\it degenerate}. We say that $\delta=\delta_X=2\dim X+1-\dim\Sec(X)$ is
the {\it defect\/}  of $\Sec(X)$ or  the {\it secant defect}\/ of $X$.

Our primary goal is to use these notions for the projectivisation of minimal orbits, i.e., if $X=\BP(\omin(\vlb))=G{\cdot}[v_\lb]\subset \BP(\vlb)$, where $[v_\lb]$ is the image of
$v_\lb$ in $\BP(\vlb)$. 
\\ \indent
{\it\bfseries Notation}: if $v\in \vlb$, then $[v]$ is a point in $\BP(\vlb)$, whereas
$\lg v\rg$ is a 1-dimensional subspace of $\vlb$. More generally, $\lg v_1,\dots, v_k\rg$ is the linear
span of $v_1,\dots,v_k\in\vlb$. 

For a projective variety $X\subset \BP(V)$, let $\widehat X\subset V$ denote the closed cone over $X$. 
That is, if $\pi:V\setminus\{0\}\to \BP(V)$ is the canonical map, then $\widehat X=\pi^{-1}(X)\cup\{0\}$.
Then $\widehat {\Sec(X)}$ is called the {\it conical secant variety} of $\widehat X$. We also write 
$\CS(\widehat X)$ or $\CS(\pi^{-1}(X))$ for this affine version of the secant variety.
The formula for the defect of $\Sec(X)$ in terms of affine varieties reads 
\beq          \label{eq:aff-defect}
    \delta_X=2\dim\widehat X- \dim \CS(\widehat X) .
\eeq
If $Y\subset \vlb$ is a subvariety, then we set
$\eus K(Y)=\ov{\bbk^*{\cdot} Y}\subset \vlb$. By definition, $\eus K(Y)$ is a closed cone in $\vlb$. Here 
$\dim\eus K(Y)=\dim Y$ if and only if $Y$ is already conical. Otherwise, $\dim\eus K(Y)=\dim Y+1$.
Clearly, if $Y$ is $G$-stable, then so is $\eus K(Y)$.

The conical secant variety of a minimal orbit has a simple explicit description.
Let $\mu$ be the lowest weight of $\vlb$ and $v_\mu$ a lowest weight vector. If $B^-$ is 
the Borel subgroup opposite to $B$ (i.e., $B\cap B^-=T$), then $v_\mu$ is just a $B^-$-eigenvector.
If $w_0$ is the longest element of the Weyl group $W=N_G(T)/T$, then $\mu=w_0(\lb)$ and
$\lb^*:=-w_0(\lb)$ is the highest weight of $(\vlb)^*$~\cite[Ch.\,3, \S\,2.7]{t41}. Hence 
$\lb=-\mu$\/ if and only if $\vlb$ is a self-dual $G$-module.

The following assertion is not really new. Part (1) is implicit in \cite{chap} and part~(2) is stated 
in~\cite[p.\,536]{kaji99}. 
However, to the best of my knowledge,  an accurate proof is not easy to locate in the literature,
For this reason, we have chosen to give a complete proof,
which, incidentally, does not invoke Terracini's lemma.

\begin{thm}       \label{thm:sec-min-orb}
With notation as above, set $h_{\lb}=v_\lb+v_{-\lb^*}=v_\lb+v_\mu\in \vlb$. Then 
\begin{itemize}
\item[(1)] \ $ \Sec(G{\cdot}[v_\lb])=\ov{G{\cdot}[h_{\lb}]}=\BP(\eus K(G{\cdot}h_{\lb}))$
and $\CS(G{\cdot}v_\lb)=\eus K(G{\cdot}h_{\lb})$;
\item[(2)] \ the defect of\/ $\Sec(G{\cdot}[v_\lb])$ equals 
$\delta_{G{\cdot}[v_\lb]}=\dim (\g{\cdot}v_{\mu}\cap \g{\cdot}v_\lb)$.
\end{itemize}
\end{thm}
\begin{proof}
(1) \ It is clear that $\BP(\eus K(G{\cdot}h_{\lb}))=\ov{G{\cdot}[h_{\lb}]}$.
\\
Note that $[h_\lb]\in \BP^1_{[v_\mu],[v_{\lb}]}$, hence $[h_\lb]\in \Sec(G{\cdot}[v_\lb])$. Moreover,
since $\lb\ne\mu$, the $T$-orbit of $[h_\lb]$ is dense in $\BP^1_{[v_\mu],[v_{\lb}]}$. Therefore, in order to prove that $G{\cdot}[h_{\lb}]$ is dense in $\Sec(G{\cdot}[v_\lb])$, it suffices to show that a generic 
secant line $\BP^1_{x,y}$ is $G$-conjugate to $\BP^1_{[v_\mu],[v_{\lb}]}$.

Indeed, because $G{\cdot}[v_\lb]$ is homogeneous, any secant line $\BP^1_{x,y}$ is $G$-conjugate to 
a line of the form $\BP^1_{x', [v_\lb]}$. Next, by the Bruhat decomposition, $C_{w_0}:=Bw_0 U$ is a 
dense open subset of $G$ (the big cell in $G$). Hence $C_{w_0}{\cdot}[v_\lb]=B{\cdot}[v_\mu]$ is dense in $G{\cdot}[v_\lb]$. Since the isotropy group of $[v_\lb]$ contains $B$, we see that if 
$x'\in B{\cdot}[v_\mu]$, then the secant line $\BP^1_{x', [v_\lb]}$ is $B$-conjugate to $\BP^1_{[v_\mu], [v_\lb]}$, 
as required.

(2) \ To compute the defect, we have to compare $\dim G{\cdot}[h_\lb]$ and $\dim G{\cdot}[v_\lb]$. To 
this end, we use the triangular decomposition $\g=\ut^-\oplus\te\oplus\ut$, where $\be=\te\oplus\ut$ and 
$\be^-=\te\oplus\ut^-$. One has $\ut{\cdot}v_\lb=0$ and $\ut^-{\cdot}v_\mu=0$. Hence
\[
   \g{\cdot}h_\lb=\ut^-{\cdot}v_\lb + \ut{\cdot}v_\mu + \te{\cdot}h_\lb .
\]
Here $\dim\ut^-{\cdot}v_\lb=\dim\ut{\cdot}v_\mu=(\dim G{\cdot}v_\lb)-1=\dim G{\cdot}[v_\lb]$. 
Set $p=\dim(\ut{\cdot}v_{\mu}\cap \ut^-{\cdot}v_\lb)$. Then
$\dim(\ut^-{\cdot}v_\lb + \ut{\cdot}v_\mu)=2\dim G{\cdot}[v_\lb] -p$.
If $t\in\te$, then $t{\cdot}h_\lb=\lb(t)v_\lb+ \mu(t)v_{\mu}$. Hence 
\beq    \label{eq:dim-te-h}
\te{\cdot}h_\lb 
=\begin{cases}  \lg v_\lb,v_\mu\rg, & \text{ if } \lb \ne -\mu, \\ \lg v_\lb-v_\mu\rg, & \text{ if } \lb = -\mu .
\end{cases} 
\eeq
\indent
{\bf (I)} \ Suppose that $v_\mu\not\in \ut^-{\cdot}v_\lb$. Since $\g{\cdot}v_\lb=\be^-{\cdot}v_\lb=
\lg v_\lb\rg\oplus \ut^-{\cdot}v_\lb$, we also have  $v_\mu\not\in \g{\cdot}v_\lb$.
(By symmetry, this also means that $v_\lb\not\in \ut{\cdot}v_\mu$, etc.) Then
\beq     \label{eq:dim-I}
  \g{\cdot}h_\lb=(\ut^-{\cdot}v_\lb + \ut{\cdot}v_\mu) \oplus \te{\cdot}h_\lb 
\eeq
and $\g{\cdot}v_{\mu}\cap \g{\cdot}v_\lb=\be{\cdot}v_{\mu}\cap \be^-{\cdot}v_\lb=\ut{\cdot}v_{\mu}\cap \ut^-{\cdot}v_\lb$.
Here one has two possibilities.

(a) If $\lb\ne-\mu$, then it follows from~\eqref{eq:dim-te-h} and \eqref{eq:dim-I} that 
$\dim\g{\cdot}h_\lb=2\dim G{\cdot}[v_\lb] -p +2$. Moreover, here 
$h_\lb\in\te{\cdot}h_\lb\subset \g{\cdot}h_\lb$. Hence the orbit $G{\cdot}h_\lb$ is conical and
\\
\centerline{$\dim\g{\cdot}[h_\lb]=\dim\g{\cdot}h_\lb-1=2\dim G{\cdot}[v_\lb] +1-p$.}
\\
Thus, here \ $p=\dim(\g{\cdot}v_{\mu}\cap \g{\cdot}v_\lb)$ is the defect of $\Sec(G{\cdot}[v_\lb])$.

(b) If $\lb=-\mu$, then $\dim\g{\cdot}h_\lb=2\dim G{\cdot}[v_\lb] -p +1$. Here 
$h_\lb\not\in\te{\cdot}h_\lb=\lg v_\lb-v_\mu\rg$. Hence the orbit $G{\cdot}h_\lb$ is {\bf not} conical and
$\dim\g{\cdot}[h_\lb]=\dim\g{\cdot}h_\lb=2\dim G{\cdot}[v_\lb] +1-p$, with the same conclusion as in (a).

{\bf (II)} \ Suppose that $v_\mu\in \ut^-{\cdot}v_\lb$ (equivalently, $v_\lb\in \ut{\cdot}v_\mu$). Then
\[
  \g{\cdot}h_\lb= \ut^-{\cdot}v_\lb + \ut{\cdot}v_\mu 
\]
and the orbit $G{\cdot}h_\lb$ is conical. Therefore, 
\\
\centerline{$\dim\g{\cdot}[h_\lb]=2\dim G{\cdot}[v_\lb] -1-p=2\dim G{\cdot}[v_\lb]+1-(p+2)$.} 
\\
But in this case,
$\g{\cdot}v_{\mu}\cap \g{\cdot}v_\lb=(\ut{\cdot}v_{\mu}\cap \ut^-{\cdot}v_\lb)\oplus
\lg v_\lb,v_\mu\rg$ and $\dim(\g{\cdot}v_{\mu}\cap \g{\cdot}v_\lb)=p+2$. Thus, one again obtains 
the required value for the defect.
\end{proof}

\begin{rmk}    \label{rem:defect}
{\bf 1)} Our proof of part~(2) shows that if $v_\mu\not\in \g{\cdot}v_\lb$, then
$\g{\cdot}v_{\mu}\cap \g{\cdot}v_\lb=\ut{\cdot}v_{\mu}\cap \ut^-{\cdot}v_\lb$.
It was also noticed that if $v_\mu\not\in \g{\cdot}v_\lb$ and 
$\mu=-\lb$, then the orbit $G{\cdot}h_\lb$ is not conical. Actually, 
we prove below that in the last case $G{\cdot}h_\lb$ is closed in $\vlb$.
\\ \indent
{\bf 2)} If $h'_\lb=av_\lb+bv_\mu$ with $a,b\ne 0$, then $\eus K(G{\cdot}h_\lb)=\eus K(G{\cdot}h'_\lb)$ and
$G{\cdot}[h_\lb]=G{\cdot}[h'_\lb]$.
\end{rmk}

\subsection{Some invariant-theoretic consequences}
\label{subs:IT-sledstviya}
If a reductive group $G$ acts on an affine variety $Z$, then the algebra of invariants
$\bbk[Z]^G$ is finitely-generated and $Z\md G:=\spe(\bbk[Z]^G)$ is the categorical quotient, see~\cite{vp89}. The
inclusion $\bbk[Z]^G\subset\bbk[Z]$ yields the morphism $\pi_Z:Z\to Z\md G$, which is onto. 
Each fibre $\pi^{-1}_Z(\xi)$, $\xi\in Z\md G$, contains a unique closed orbit, hence $Z\md G$ 
parametrises the closed $G$-orbits in $Z$. If $Z=\VV$ is a $G$-module, then the fibre
$\N_G(\VV)=\pi_\VV^{-1}(\pi_\VV(0))$ is called the {\it null-cone}. It is also true that
\[
    \N_G(\VV)=\{v\in\VV \mid \ov{G{\cdot}v}\ni 0\}.
\]
Recall that $\mu=-\lb^*$ is the lowest weight of $\vlb$. In the proof of Theorem~\ref{thm:sec-min-orb}, 
we considered the condition that $v_{\mu}\not\in\g{\cdot}v_\lb$. Actually, if $\g$ is simple, then this condition is not satisfied 
only for two cases.

\begin{lm}     \label{lm:spn-vp1}
Suppose that $\g$ is simple and $v_{\mu}\in\g{\cdot}v_\lb$. Then either 
$\g=\slno$ and $\{\lb,\lb^*\}=\{\varpi_1,\varpi_n\}$ or $\g=\spn$ and $\lb=\lb^*=\varpi_1$. 
\end{lm}
\begin{proof}
Here $v_\mu\in\ut^-{\cdot}v_\lb$ and 
$\lb-\mu=\lb+\lb^*\in \Delta^+$. Hence a dominant root can be written as a sum of two dominant weights. This is only possible if $\lb+\lb^*=\theta$ and $\theta$ is not fundamental, i.e., if
$\g=\slno$ and $\theta=\varpi_1+\varpi_n$ or $\g=\spn$ and $\theta=2\varpi_1$.
\end{proof}

\begin{prop}     \label{prop:krit-Luna}
If\/ $\vlb$ is self-dual and $v_{\mu}\not\in\g{\cdot}v_\lb$, then the orbit $G{\cdot}h_\lb\subset\vlb$ 
is closed.
\end{prop}
\begin{proof}
If $\vlb$ is self-dual, then $\mu=-\lb$ and $h_\lb=v_\lb+v_{-\lb}$. Let $H$ be a reductive subgroup of 
$G^{h_\lb}$. By Luna's criterion, $G{\cdot}h_\lb$ is closed {\sl if and only if} $N_G(H){\cdot}h_\lb$ is 
closed, see~\cite[Theorem\,6.17]{vp89}.

Take $H=(G^{h_\lb}\cap T)^0$, the identity component of $G^{h_\lb}\cap T$. It is a torus of 
codimension~1 in $T$. Then $N_G(H)^0=Z_G(H)$ and $Z_G(H)$ is either $T$ or $H{\cdot}SL_2$.
The latter occurs only if $\lb$ is proportional to a root. In both cases, one has to prove that
the $Z_G(H)/H$-orbit of $h_\lb$ is closed. For $Z_G(H)/H\simeq\bbk^*$,  the assertion follows from
\cite[Prop.\,6.15]{vp89}.

If $Z_G(H)/H\simeq SL_2$, then $h_\lb$ is the sum of a highest and lowest weight vectors in a
simple $SL_2$-submodule $\sfr_j\subset\vlb$, where $\dim\sfr_j=j+1$. (If $\mu\in\Delta^+$ is the root of 
$SL_2\subset Z_G(H)$, then $j=(\lb,\mu^\vee)$.) The assumption that 
$v_\mu\not\in\g{\cdot}v_\lb$ guarantees us
that $j>1$ and therefore $SL_2{\cdot}h_\lb$ is closed in $\sfr_j$.
\end{proof}
\noindent
Using Theorem~\ref{thm:sec-min-orb}(1), 
we can  characterise the cases in which $\CS(\omin(\vlb))= \vlb$.

\begin{thm}       \label{thm:kogda-nevyrozhd}
Let $\vlb$ be a simple $G$-module.
\begin{itemize}
\item[\sf (i)] \ If\/ $\vlb\not\simeq(\vlb)^*$, then $\CS(\omin(\vlb))\subset\N_G(\vlb)$.
Moreover, if\/ $\vlb\md G\ne\{pt\}$ (i.e., $\bbk[\vlb]^G\ne \bbk$), then
$\CS(\omin(\vlb))\subset\N_G(\vlb)\subsetneqq \vlb$;
\item[\sf (ii)] \ If\/ $\vlb\simeq(\vlb)^*$ and $\dim\vlb\md G> 1$, then
$\CS(\omin(\vlb))\subsetneqq \vlb$;
\item[\sf (iii)] \ $\CS(\omin(\vlb))=\vlb$ if and only if\/ $\g{\cdot}v_\lb+\g{\cdot}v_\mu=\g$.
\end{itemize}
\end{thm}
\begin{proof}
{\sf (i)} \ If $\vlb\not\simeq(\vlb)^*$, then $\lb\ne\lb^*$ and 
the weights $\lb$ and $\mu=-\lb^*$ belong to an open
halfspace of $\te^*_\BQ$. Therefore $h_\lb\in \N_G(\vlb)$ and hence 
$\eus K(G{\cdot}h_\lb)\subset \N_G(\vlb)$. If $\vlb\md G\ne\{pt\}$, then $\N_G(\vlb)$ is a proper subvariety of $\vlb$.

{\sf (ii)} \ If $\vlb\simeq(\vlb)^*$, then $G{\cdot}h_\lb=\ov{G{\cdot}h_\lb}$ and the $G$-variety 
$Z:=\eus K(G{\cdot}h_\lb)$ contains a $1$-parameter family of closed $G$-orbits as a dense subset.
Hence $\dim Z\md G=1$. 
Therefore, if $\dim\vlb\md G> 1$, then $Z$ is a proper subvariety of $\vlb$.

{\sf (iii)} \ By Theorem~\ref{thm:sec-min-orb} and \eqref{eq:aff-defect}, $\dim(\g{\cdot}v_\lb+\g{\cdot}v_\mu)=2\dim \g{\cdot}v_\lb-\delta=
\dim\CS(\omin(\vlb))$.
\end{proof}

It follows that the irreducible representations of {\bf simple} algebraic groups $G$ such that 
$\CS(\omin(\vlb))= \vlb$ can be extracted from the lists of
\\ \indent
(a) \ the non-self-dual $G$-modules $\vlb$ with $\bbk[\vlb]^G=\bbk$, and
\\ \indent
(b) \ the self-dual  $G$-modules $\vlb$ with $\dim\vlb\md G\le 1$.

\noindent
All such representations occur in the "Summary Table" in~\cite{vp89}, where the irreducible
representations of simple Lie groups with polynomial algebras of invariants are listed. If $\vlb$ satisfies 
either (a) or (b), then the condition $\eus K(G{\cdot}h_\lb)=\vlb$ means that $h_\lb\in\vlb$ is a {\it point of 
general position} in the sense of~\cite{ag72}. Since Table~1 in~\cite{ag72} explicitly indicates such points for all representations in question, one easily obtains the following tables, where  the relevant
highest weights $\lb$ and the defect $\delta=\delta_{\BP(\omin(\lb))}$ are given. If $\vlb$ is
self-dual, then we point out whether it is orthogonal or symplectic. The numbering of simple roots 
follows~\cite{t41} and we write $\varpi_i$ in place of $\varpi_{\ap_i}$.

\begin{table}[ht]
\caption{The representations with $\CS(\omin(\vlb))=\vlb$: serial cases}   \label{table:serial}
\begin{tabular}{ >{$}c<{$}| >{$}c<{$} >{$}c<{$} >{$}c<{$} >{$}c<{$} cc}
\g & \lb & \dim\vlb & \dim\omin(\vlb) & \delta & self-dual & $\dim\vlb\md G$ \\ \hline\hline
\GR{A}{n}, n\ge1 & \varpi_1,\varpi_n & n{+}1 & n{+}1 & n{+}1 & {\it no} & 0 \\
\GR{B}{n}, n\ge 2 & \varpi_1 & 2n{+}1 & 2n & 2n{-}1 & orth. & 1 \\
\GR{C}{n}, n\ge 2 & \varpi_1 & 2n & 2n & 2n & sympl. & 0 \\
\GR{D}{n}, n\ge 5 & \varpi_1 & 2n & 2n{-}1 & 2n{-}2 & orth. & 1\\
\hline
\end{tabular}
\end{table}

\begin{table}[ht]
\caption{The representations with $\CS(\omin(\vlb))=\vlb$: sporadic cases}   \label{table:sporad}
\begin{tabular}{ >{$}c<{$}| >{$}c<{$} >{$}c<{$} >{$}c<{$} >{$}c<{$} cc}
\g & \lb & \dim\vlb & \dim\omin(\vlb) & \delta & self-dual & $\dim\vlb\md G$ \\ \hline\hline
\GR{A}{1} & 2\varpi_1 & 3 & 2 & 1 & orth. & 1 \\
                & 3\varpi_1 & 4 & 2 & 0 & sympl. & 1 \\
\GR{A}{3} & \varpi_2 & 6 & 5 & 4 & orth. & 1 \\
\GR{A}{4} & \varpi_2,\varpi_3 & 10 & 7 & 4 & {\it no} & 0 \\
\GR{A}{5} & \varpi_3 & 20 & 10 & 0 & sympl. & 1 \\
\GR{B}{3} & \varpi_3 & 8 & 7 & 6 & orth. & 1 \\
\GR{D}{4} & \varpi_1,\varpi_3,\varpi_4 & 8 & 7 & 6 & orth. & 1\\
\GR{B}{4} & \varpi_4 & 16 & 11 & 6 & orth. & 1 \\
\GR{C}{3} & \varpi_3 & 14 & 7 & 0 & sympl. & 1 \\
\GR{D}{5} & \varpi_4,\varpi_5 & 16 & 11 & 6 & {\it no} & 0 \\
\GR{D}{6} & \varpi_5,\varpi_6 & 32 & 16 & 0 & sympl. & 1 \\
\GR{E}{7} & \varpi_1 & 56 & 28 & 0 & sympl. & 1 \\
\GR{G}{2} & \varpi_1 & 7 & 6 & 5 & orth. & 1 \\
\hline
\end{tabular}
\end{table}

\begin{rmk}
The simple $G$-modules with $\delta>0$ have been listed in~\cite{kaji99}. The representations in 
Tables~\ref{table:serial} and~\ref{table:sporad} that do not occur in \cite{kaji99} are those with 
$\delta=0$.
\end{rmk}

We can also provide a quick approach to the classification in~\cite{kaji99}, which is outlined below.
For $\ap\in\Pi$, let $[\theta:\ap]$ denote the coefficient of $\ap$  in the expression of $\theta$ via $\Pi$. 
Set $r_\ap=\frac{(\theta,\theta)}{(\ap,\ap)}\in \{1,2,3\}$. In particular, $r_\ap=1$ for all $\ap$ in the 
simply-laced case.

\begin{lm}         \label{lm:lb-theta=1}
Let $\lb$ be a dominant weight. Then $(\lb,\theta^\vee)=1$ if and only if $\lb=\varpi_\ap$ is fundamental 
and $[\theta:\ap]=r_\ap$. 
\end{lm}
\begin{proof}
Write $\lb=\sum_{i=1}^l a_i\varpi_i$ and $\theta=\sum_{i=1}^l n_i\ap_i$, i.e., $n_i=[\theta:\ap_i]$. 
(Here $l=\rk\,\g$.) Then
\[
   (\lb,\theta^\vee)=\frac{2}{(\theta,\theta)}\sum_{l=1}^l a_i n_i (\varpi_i,\ap_i)=
   \sum_{l=1}^l a_i n_i \frac{(\ap_i,\ap_i)}{(\theta,\theta)}=\sum_{l=1}^l a_i \frac{n_i}{r_i} ,
\]
where $r_i=r_{\ap_i}$. Since $\theta$ is a long root, we have $n_i/r_i\in\BN$ for all $i$. 
The assertion follows.
\end{proof}
 
\begin{thm}        \label{thm:delta>0}
Suppose that $\delta=\delta_{G{\cdot}[v_\lb]}>0$. Then 
\begin{itemize}
\item[\sf (i)] \ $(\lb,\theta^\vee)=(\lb^*,\theta^\vee)\in\{1,2\}$; 
\item[\sf (ii)] \ If $(\lb,\theta^\vee)=2$, then $\delta=1$ and either $\lb=\lb^*=\theta$ (the adjoint 
representation) or $\g=\slno$ and $\{\lb,\lb^*\}=\{2\varpi_1,2\varpi_n\}$.
\end{itemize}
\end{thm}
\begin{proof}
{\sf (i)} \ If $\delta=\dim(\g{\cdot}v_\lb\cap\g{\cdot}v_{-\lb^*})>0$, 
then there are $\nu,\nu'\in\Delta^+\cup\{0\}$ such that $\lb-\nu=-\lb^*+\nu'$. Hence
$\lb+\lb^*=\nu+\nu'$ is a dominant weight in the root lattice. Since $\theta=\theta^*$,
we have $(\lb,\theta^\vee)=(\lb^*,\theta^\vee)$ is a positive integer. On the other hand,
$(\nu,\theta^\vee), (\nu',\theta^\vee)\in\{0,1,2\}$. This implies that
\[
     2(\lb,\theta^\vee)=(\lb+\lb^*,\theta^\vee)=(\nu+\nu',\theta^\vee)\in \{2,4\} .
\]
\indent
{\sf (ii)} \ If $(\lb,\theta^\vee)=2$, then $(\nu,\theta^\vee)=(\nu',\theta^\vee)=2$. Hence $\nu=\nu'=\theta$  and $\lb+\lb^*=2\theta$. If $\theta$ is a multiple of a fundamental weight, then the only possibility is
$\lb=\lb^*=\theta$. For $\g=\slno$, one has $\theta=\varpi_1+\varpi_n$, which yields one more
possibility. In either case, $\lb-\theta=-\lb^*+\theta$ is the only weight occurring in
$\g{\cdot}v_\lb\cap\g{\cdot}v_{-\lb^*}$, i.e., $\delta=1$.
\end{proof}

\begin{rmk}
It is worth noting that not every $\lb=\varpi_\ap$ such that $[\theta:\ap]=r_\ap$ gives rise to the 
$G$-module $\vlb$ with $\delta_{G{\cdot}[v_\lb]}>0$. For $\g=\slno$ or $\spn$, all simple roots $\ap$ 
satisfy that condition, while one has $\delta_{G{\cdot}[v_\lb]}=0$ in many cases. 
We can prove the following assertion, whose proof and applications
will appear in a forthcoming publication:

\emph{If $\lb$ is dominant and $(\lb,\theta^\vee)=1$, then
$\delta_{G{\cdot}[v_\lb]}>0$ if and only 
if\/ $\lb+\lb^*-\theta\in (\Delta^+\setminus\{\theta\})\cup\{0\}$. 
Moreover, in this case $\delta\ge 2$.}
\end{rmk}

\subsection{The minimal nilpotent orbit in $\g$}
\label{subs:omin}
For a simple Lie algebra $\g\simeq\VV_\theta$, we usually write $\omin$ in place of $\omin(\g)$ and say that
$\omin\subset\g$ is the {\it minimal nilpotent orbit}. If $\gamma\in\Delta$ is a long root and
$e_\gamma\in\g^\gamma$, then $\omin=G{\cdot}e_\gamma$. (In the simply-laced case, all roots are 
assumed to be long.)

Let $\{e,h,f\}$ be an $\tri$-triple with $e\in\omin$. Such an $\tri$-triple is said to be {\it minimal}.  
For the minimal $\tri$-triples, the $(\BZ,h)$-grading of $\g$ has the following structure:
\beq       \label{eq:Z-h-grad}
\text{$\g=\textstyle \bigoplus_{j=-2}^2 \g(i)$, \ $\g(2)=\lg e\rg$, \ and \ $\g(-2)=\lg f\rg$.}
\eeq
Since $\dim\g^e=\dim\g(0)+\dim\g(1)$ and $\g^h=\g(0)$, we have 
$\dim\omin=\dim\g(1)+2$ and $\dim G{\cdot}h=2\dim\g(1)+2$. Hence
\beq       \label{eq:dim-Gh}
      \dim G{\cdot}h= 2\dim\omin -2 .
\eeq
Set $\es=\lg e,h,f\rg\simeq \tri$. Let $\sfr_i$ denote a simple $\tri$-module of dimension $i+1$. It follows from~\eqref{eq:Z-h-grad} that 
\[
    \g\vert_\es=m_0\sfr_0+m_1\sfr_1+\sfr_2 ,
\]
where $m_0=\dim\g(0)-1$ and $m_1=\dim\g(1)$. The  $G$-orbit $\omin$ is characterised by the 
property that $\g\vert_\es$ contains no $\sfr_j$ with $j\ge 3$ and $\sfr_2$ occurs only once. 
A formulation that does not invoke  $\tri$-triples and $\BZ$-gradings is
\beq       \label{eq:ohne-Z-grad}
   \text{ $e\in\omin$ if and only if \ $\Ima((\ad e)^2)=\lg e\rg$.}
\eeq
Hence $\hot(\omin)=2$. However, if $\g\ne\tri$ or $\mathfrak{sl}_3$, then there are other 
$G$-orbits of height $2$.
\\ \indent
Let $\theta\in \Delta^+$ be the highest root. Then $e_\theta\in \g^\theta$ is a highest weight vector
and $f_{\theta}\in \g^{-\theta}$ is a lowest weight vector. Under a suitable adjustment, three elements
$e_\theta, h_\theta=[e_\theta, f_\theta], f_\theta$ form an $\tri$-triple. Here $e_\theta +f_\theta$ is 
conjugate to $h_\theta$ in $\lg e_\theta, h_\theta, f_\theta\rg\simeq \tri$ and hence in $\g$. Therefore, it 
follows from Theorem~\ref{thm:sec-min-orb}(1)  that
\beq         \label{eq:secant-omin}
   \Sec(\BP(\omin))=G{\cdot}[h] \ \text{ and } \CS(\omin)=\eus K(G{\cdot}h) ,
\eeq
where $h$ is the characteristic a minimal $\tri$-triple.
This recovers the main result of~\cite{koy99}. Since the orbit $G{\cdot}h$ is not conical, 
equality~\eqref{eq:dim-Gh} can be written as $\dim G{\cdot}[h]=2\dim \BP(\omin)$, which again shows 
that here $\delta_{\BP(\omin)}=1$. It is also easily seen that $[\g,e_\theta]\cap [\g,f_\theta]=\lg h_\theta\rg$.

\begin{rmk}
The $G$-modules $\vlb$ such that $\CS(\omin(\vlb))=\vlb$ and $\delta=0$ are related to the minimal 
nilpotent orbits of the exceptional Lie algebras. If $\g$ is exceptional, then $\theta$ is fundamental and
the $(\BZ,h)$-grading \eqref{eq:Z-h-grad} has the property that $\es=[\g(0),\g(0)]$ is simple. Here
$\rk\,\es=\rk\,\g-1$, $\g(1)$ is a simple (symplectic) $\es$-module, and the five pairs $(\es, \g(1))$ 
are exactly the five items in Table~\ref{table:sporad} with $\delta=0$.
\end{rmk}

\section{On the intersections of $\omin$ with $\g_0$ and $\g_1$} 
\label{sect:intersec}

\noindent
Given $\sigma\in\mathsf{Inv}(\g)$ and the related $\BZ_2$-grading
$\g=\g_0\oplus\g_1$, let $p_i:\g\to\g_i$  ($i=0,1$)
be the corresponding projections. That is, 
\[
  p_0(x)=\frac{x+\sigma(x)}{2} \ \text{ and } \ p_1(x)=\frac{x-\sigma(x)}{2}, \quad x\in\g .
\]
Both projections are $G_0$-equivariant. Set $\vp=p_0\vert_{\bomin}$ and $\psi=p_1\vert_{\bomin}$.
To a great extent, properties of $\psi$ and $\vp$  
depend on the varieties $\omin\cap\g_0$ and $\omin\cap\g_1$. We begin with a simple observation.

\begin{prop}                  \label{prop:finite-proj}
Let $\g=\g_0\oplus\g_1$ be a $\BZ_2$-grading.
\begin{itemize}
\item[\sf (i)] Suppose that $\omin\cap\g_0=\varnothing$. Then 
$\psi:\bomin\to \psi(\bomin)=\ov{\psi(\omin)}$ is finite.  In particular, $\dim\omin=\dim\psi(\omin)$.
Furthermore, if $x\in\psi(\bomin)$, then 
\[
   \#\psi^{-1}(x)=1 \Longleftrightarrow\psi^{-1}(x)=\{x\} \Longleftrightarrow x\in\omin\cap\g_1 .
\]
\item[\sf (ii)] Suppose that $\omin\cap\g_1=\varnothing$. Then 
$\vp:\bomin\to \vp(\bomin)=\ov{\vp(\omin)}$ is finite. In particular, $\dim\omin=\dim\vp(\omin)$.
Furthermore, 
if $x\in\vp(\bomin)$, then 
\[
   \#\vp^{-1}(x)=1 \Longleftrightarrow \vp^{-1}(x)=\{x\} \Longleftrightarrow x\in\omin\cap\g_0 .
\]
\end{itemize}
\end{prop}
\begin{proof}
{\sf(i)} \ Here $\psi^{-1}(\psi(0))=\bomin\cap\g_0=\{0\}$. Since both $\bomin$ and $\ov{\psi(\omin)}$ are
$\bbk^*$-stable (=\,conical) and $\psi$ is $\bbk^*$-equivariant, this implies 
that $\psi$ is finite. 
\\ \indent
Let $x\in \omin\cap\g_1$. Assume that  $y+x\in \psi^{-1}(x)$ for some nonzero $y\in\g_0$. Since 
$\sigma(\omin)=\omin$, we have $\sigma(y+x)=y-x\in\omin$ and $-y+x\in \psi^{-1}(x)$. As the plane $\gP=\lg x,y\rg$
contains three different lines from $\bomin$, we conclude by Lemma~\ref{lm:omin-plane} that 
$y\in \omin\cap\g_0$. A contradiction! Hence $\vp^{-1}(x)=\{x\}$.
\\ \indent
Conversely, if $x\in \psi(\omin)\setminus \omin$, then $x=\psi(y+x)$ for some nonzero $y\in\g_0$. In this case, we have $-\sigma(x+y)=-y+x\in \psi^{-1}(x)$, i.e., $\#\psi^{-1}(x)> 1$.
\\ \indent
{\sf(ii)} \ Similar.
\end{proof}

\noindent
We prove below that whenever $\psi$ or $\vp$ is finite, the degree of these morphisms equals~$2$. 
Proposition~\ref{prop:finite-proj} is fully symmetric relative to $\psi$ and $\vp$, but such symmetry fails 
below. We shall see that $\omin\cap\g_1$ is more important for the properties of $\Ima(\vp)$ and 
$\Ima(\psi)$.  

Our next goal is to classify all pairs $(\g,\sigma)$ such that $\omin\cap\g_0=\varnothing$ or
$\omin\cap\g_1=\varnothing$.

\begin{ex}   \label{ex:sln-max}
(1) Take $\sigma=\vth_{\sf max}$ for $\g=\sln$. Then $\vth_{\sf max}(A)=-A^t$ 
for a matrix $A\in\sln$. Here
\[
    \omin=\{A\in\sln \mid \rk\, A=1\} \ \text{ and } \ \dim(\omin)=2n-2 .
\]
Since $\rk\,A$ is even for any $A\in\son=\g_0$, one has $\omin\cap\g_0=\varnothing$ here. In this
case 
\\[.6ex]
   \centerline{$\psi(A)=\displaystyle\frac{A+A^t}{2}$ and hence \ 
   $\psi(\bomin)\subset  \{B\in\sln\mid B=B^t \ \& \  \rk\,B\le 2\}$.}
\\[.6ex]
Recall that a  {\it symmetric determinantal variety\/} is $\mathsf{Sym}_k(n)=\{B\in\gln\mid B=B^t \ \& \  \rk\,B\le k\}$.
It is irreducible, normal, and $\dim \mathsf{Sym}_k(n)=k(2n-k+1)/2 $~\cite[6.2.5]{smt}. Set 
$\mathsf{Sym}^0_k(n)=\mathsf{Sym}_k(n)\cap\sln$. Using the interpretation of $\mathsf{Sym}_k(n)$
as the categorical quotient  $n\eus V\md O(\eus V)$, where $\dim \eus V=k$~\cite[Thm.\,12.1.7.2]{smt}, one readily realises that there is a normal irreducible hypersurface $\eus D\subset n\eus V$ such that
$\eus D\md O(\eus V)=\mathsf{Sym}^0_k(n)$. Hence the latter is also irreducible and normal, with 
dimension one less. Since $\psi(\bomin)\subset \mathsf{Sym}^0_2(n)$ and their dimensions are equal,
we obtain $\psi(\bomin)=\mathsf{Sym}^0_2(n)$. In particular, $\psi(\bomin)$ is normal.

(2) Take $\sigma=\vth_{\sf max}$ for $\g=\spn$. Under a suitable choice of symplectic form on 
$\bbk^{2n}$, a matrix $M\in\spn$ has the  presentation 
$M=\begin{pmatrix}  A & B \\ C & -A^t \end{pmatrix}$, 
where $A,B,C$ are $n\times n$ matrices, $B=B^t$, and $C=C^t$. Then  
\[
\text{ $\g_0=\left\{\begin{pmatrix}  A & 0 \\ 0 & -A^t \end{pmatrix}\right\}\simeq \gln$ \ and \  
$\g_1=\left\{\begin{pmatrix}  0 & B \\ C & 0 \end{pmatrix}\right\}\simeq \mathsf{Sym}(n)\times  \mathsf{Sym}(n)$. }
\]
One similarly has $\omin=\{M\in\spn \mid \rk\, M=1\}$ and $\rk\, M$ is even for any 
$M\in\g_0$. Hence $\omin\cap\g_0=\varnothing$ as well. Here $\dim\omin=2n$ and 
$\psi(\bomin)\simeq \mathsf{Sym}_1(n)\times\mathsf{Sym}_1(n)$. It follows that $\psi(\bomin)$ is also
normal here.
\end{ex}
It appears that these are the only instances, where $\omin\cap\g_0=\varnothing$.
\begin{thm}    \label{thm:g0-empty}
One has $\omin\cap\g_0=\varnothing$ if and only if\/ $\sigma=\vth_{\sf max}$ and $\g$ is either 
$\sln$ or $\spn$. 
\end{thm}
\begin{proof}
{\it \bfseries ``if'' part}: See the preceding example. 

{\it \bfseries ``only if'' part}:  
\\ \indent
\textbullet\quad If $\sigma$ is inner, then we take $\te\subset\g_0$ and consider the natural
partition of the root system $\Delta=\Delta(\g_0)\sqcup\Delta(\g_1)$. Clearly, if $\Delta(\g_0)$ contains
a long root, then $\omin\cap\g_0\ne\varnothing$.  And $\vth_{\sf max}$ for $\spn$ is the only inner 
involution such that all roots in $\Delta(\g_0)$ are short.

\textbullet\quad  If $\sigma\in\mathsf{Inv}(\g)$ is outer, then $\g\in\{\GR{A}{m}, \GR{D}{m}, \GR{E}{6}\}$.

\noindent
{\bf --} \  For $\g=\sltn$ and $\g_0=\spn$, we know using the rank of matrices that 
\\
\centerline{$\varnothing\ne \omin(\sltn)\cap\spn=\omin(\spn)$.} 

\noindent
{\bf --} \ Let $\g=\mathfrak{so}_{N}$ be the Lie algebra of skew-symmetric $N\times N$ matrices. Then 
\[
  \omin=\{A\in \mathfrak{gl}_{N}\mid A=-A^t, \ \rk\,A=2, \ A^2=0\}.
\] 
If $\sigma$ is outer, then $N$ is even and $(\g,\g_0)=(\mathfrak{so}_{2n+2m+2},
\mathfrak{so}_{2n+1}\dotplus\mathfrak{so}_{2m+1})$ with $n+m>0$.  It is clear that if, say,  $n>0$, then
$\mathfrak{so}_{2n+1}$ alone also contains such matrices $A$.

\noindent
{\bf --} \ For $\g=\GR{E}{6}$, there are two outer involutions, with either $\g_0=\mathfrak{sp}_8$ or
$\g_0=\GR{F}{4}$. Since $\g_0$ is simple in both cases, it suffices to verify that 
$G{\cdot}\omin(\g_0)=\omin(\g)$. Let $\{e,h,f\}\subset\g_0$ be an $\tri$-triple with 
$e\in \omin(\g_0)$. Set $\es=\lg e,h,f\rg\simeq \tri$. 

(1) For $\g_0=\mathfrak{sp}_8=\spv$, the partition of $e$ is $(2,1,\dots,1)=(2,1^6)$, i.e.,
$\VV\vert_\es=\sfr_1+6\sfr_0$.  Then $\g_0\simeq \eus S^2(\VV)$ and 
\[
    \g_0\vert_\es=\eus S^2(\sfr_1+6\sfr_0)=\sfr_2+6\sfr_1+21\sfr_0 . 
\]
Here $\g_1$ is the 4-th fundamental representation of $\mathfrak{sp}_8$ and
it is not hard to compute that 
\[
  \g_1\vert_\es=\wedge^4\VV/\wedge^2 \VV=14\sfr_1+14\sfr_0 .
\]
Hence $\g\vert_\es$ contains $\sfr_2$ only once and contains no $\sfr_j$ with $j\ge 3$.
Thus, $e\in \omin(\g)$. 

(2) For $\g_0=\GR{F}{4}$, similar calculations yield the same conclusion. 
\end{proof}

{\it \bfseries Remark.} A more conceptual (partial) argument is that if $\sat(\sigma)$ contains a black node 
corresponding to a {\bf long} root $\ap\in\Pi$, then  $\g_\ap\subset \g_0$ 
and thereby $\omin\cap\g_0\ne \varnothing$. This applies to  
$(\GR{D}{n+m+1},\GR{B}{n}\dotplus\GR{B}{m})$ with $|n-m|\ge 2$ and
$(\GR{E}{6}, \GR{F}{4})$.

To handle intersections $\omin$ with $\g_1$, we need a simple but useful assertion.

\begin{prop}    \label{prop:codim-1}
If $H$ is a spherical subgroup of $G$, then
\[
      \max_{v\in\omin}\dim H{\cdot}v\ge \dim\omin -1 .
\] 
and $H$ has a dense orbit in $\BP(\omin)\subset \BP\g$.
\end{prop}
\begin{proof}
Since $\h$ is spherical, there is a Borel subalgebra $\be\subset\g$ such that $\h+\be=\g$. Fix a 
Cartan subalgebra $\te\subset\be$ and the corresponding set of positive roots $\Delta^+$ in the root 
system of $(\g,\te)$. Let $e_\theta\in\be$ be a highest weight vector in $\be$. Then $e_\theta\in\omin$ 
and $[\be,e_\theta]=\lg e_\theta\rg$. Hence $[\g,e_\theta]=[\h,e_\theta]+\lg e_\theta\rg$. That is, 
$\dim H{\cdot}e_\theta\ge \dim\omin-1$ and the $H$-orbit of $\lg e\rg\in\BP\g$ is dense in $\BP(\omin)$.
\end{proof}

\begin{rmk}   \label{rmk:dense-orb}
It follows from this proposition that either $H$ has a dense orbit in $\omin$ or there is a dense open 
subset of $\omin$ that consists of $H$-orbits of codimension~$1$. Moreover, $H$ has a dense orbit in
$\omin$ if and only if $e_\theta\in [\h,e_\theta]$, i.e., the orbit $H{\cdot}e_\theta$ is a cone.
\end{rmk}

For any $\sigma\in\mathsf{Inv}(\g)$, $G_0$ is a spherical subgroup of $G$. Let us take a Borel 
subalgebra $\be$ such that $\g_0+\be=\g$ and a $\be$-eigenvector (i.e., a highest weight vector) 
$e_\theta\in \omin\cap\be$.

\begin{thm}   \label{thm:g1-empty}
The following conditions are equivalent:
\begin{itemize}
\item[\sf (i)] \ $G_0$ has a dense orbit in $\omin$; 
\item[\sf (ii)] \ $\omin\cap\g_1=\varnothing$;
\item[\sf (iii)] \ $\vp: \bomin\to \vp(\bomin)\subset\g_0$ is a finite morphism; 
\item[\sf (iv)] \ there is a black node $\ap$ in $\sat(\sigma)$ such that $(\ap,\theta)\ne 0$;
\item[\sf (v)] \ the orbit $G_0{\cdot}e_\theta$ is conical.
\end{itemize}
\end{thm}
\begin{proof} If $\g_0+\be=\g$, then $\te_1:=\g_1\cap\be$ is a 
Cartan subspace of $\g_1$. Taking a $\sigma$-stable Cartan subalgebra $\te\supset\te_1$, we
then make a suitable choice of $\Delta^+$, as described in Section~\ref{subs:Satake}, and 
obtain $\sat(\sigma)$. We use below 
the related description of $\te_1$ and $\te_0=\te\cap\g_0$ via $\Pi_{0,\sigma}$ and $\Pi_{1,\sigma}$.

$\mathsf{(i)}\Leftrightarrow\mathsf{(v)}$: This follows from Proposition~\ref{prop:codim-1} and 
Remark~\ref{rmk:dense-orb} with $H=G_0$.
\\ \indent
$\mathsf{(ii)}\Leftrightarrow\mathsf{(iii)}$: This follows from~\eqref{eq:closure-omin} and   
the fact that $\vp^{-1}(0)=\{0\}\cup (\omin\cap\g_1)$. Cf. also Proposition~\ref{prop:finite-proj}(ii).
\\ \indent
$\mathsf{(ii)}\Rightarrow\mathsf{(iv)}$: Recall that $\theta\in\Delta^+$ is the highest root. Then 
$\eus D(\omin)$ has the labels $(\ap^\vee,\theta)$ for $\ap\in\Pi$. That is, if $h_\theta\in\te$ is 
the $\Delta^+$-dominant characteristic for $\omin$, then $\ap(h_\theta)=(\ap^\vee,\theta)$ for all 
$\ap\in\Pi$. If $\omin\cap\g_1=\varnothing$, 
then $G{\cdot}h_\theta\cap\g_1=\varnothing$, i.e., $G{\cdot}h_\theta\cap\te_1=\varnothing$, i.e.,
$h_\theta\not\in\te_1$.
By Remark~\ref{rem:Levon}, this means that at least one thing should happen:

\textbullet \quad $(\ap,\theta)\ne 0$ for some $\ap\in\Pi_{0,\sigma}$, {\sl or}

\textbullet \quad $(\ap,\theta)\ne (\beta,\theta)$ for some $\ap,\beta\in\Pi$ joined by an arrow in
$\sat(\sigma)$.
\\
However, the second thing never happens for $\omin$. For, if $\g\ne\slno$, then
there is a unique $\ap\in\Pi$ such that $(\ap,\theta)\ne 0$; and this $\ap$ is never joined by an arrow
with another node of the Dynkin diagram. While for $\slno$, there are two such simple roots,
$\ap_1$ and $\ap_n$, and for them $(\ap_1,\theta)=(\ap_n,\theta)$.
\\ \indent
$\mathsf{(iv)}\Rightarrow\mathsf{(ii)}$:  
Suppose that $(\ap,\theta)\ne 0$ for some $\ap\in\Pi_{0,\sigma}$. Then the descriptions of $\te_0$ 
and $\te_1$ via $\sat(\sigma)$ show that $h_\theta$ cannot belong to $\te_1$.
\\ \indent
$\mathsf{(iv)}\Leftrightarrow\mathsf{(v)}$: 
If $\ap\in\Pi_{0,\sigma}$ and $(\ap,\theta)\ne 0$, then $h_\ap\in\te_0$ and
$[h_\ap,e_\theta]=(\ap^\vee,\theta)e_\theta\ne 0$, i.e., $G_0{\cdot}e_\theta$ is conical.
Conversely, if $[\g_0, e_\theta]$ contains the line $\lg e_\theta\rg$, then
$[\te_0,e_\theta]$ also does. Then the description of $\te_0$ in Section~\ref{subs:Satake} implies that 
there is $\ap\in\Pi_{0,\sigma}$ such that $[h_\ap,e_\theta]\ne 0$, i.e., $(\ap,\theta)\ne 0$.
\end{proof}

\begin{rmk}
If $\g\ne\slno$, then $\theta$ is a multiple of a fundamental weight and there is a unique $\ap\in\Pi$ such 
that $(\ap,\theta)\ne 0$. Therefore, condition {\sf (iv)} above should be verified for one node of the Dynkin 
diagram. For $\g=\sln$, one has two such simple roots.
\end{rmk}
\begin{ex}   \label{ex:6-cases}
Using condition~{\sf (iv)} and the table with Satake diagrams~\cite[Table~4]{t41}, one readily compiles 
the list of all $\sigma\in\mathsf{Inv}(\g)$ such that $\omin\cap\g_1^{(\sigma)}=\varnothing$. Here are 
the pairs $\g\supset \g_0^{(\sigma)}$ for them:

\begin{tabular}{clcl}
(1) &  $\GR{A}{2n-1}\supset \GR{C}{n}$, $n\ge 2$,  
\phantom{x\ t} 
& (4) & $\GR{C}{n}\supset \GR{C}{k}\dotplus \GR{C}{n-k}$ $(k=1,2,\dots,n-1)$, \\
(2) &  $\GR{D}{n+1}\supset \GR{B}{n}$, $n\ge 3$, & (5) & $\GR{B}{n}\supset \GR{D}{n}$, $n\ge 3$,
\\
(3)  &  $\GR{E}{6}\supset \GR{F}{4}$,    & (6) & $\GR{F}{4}\supset \GR{B}{4}$. 
\end{tabular}
\\[.6ex] 
The involutions in the first column are outer,  
whereas the second column represents inner involutions.
Note also that these six cases can be organised in three chains (read the list along the rows).
\end{ex}

\begin{rmk}       \label{rmk:six-ss-g_0}
1) For all six cases of Example~\ref{ex:6-cases}, $\g_0$ appears to be semisimple. This can be 
explained {\sl a priori}, as follows. If $\g_0\ne [\g_0,\g_0]$, then the $\BZ_2$-grading in question arises 
from a short $\BZ$-grading and $\g_1=\g(-1)\oplus\g(1)$, see Section~\ref{sect:2}. Choose a Cartan 
subalgebra $\te\subset\g_0=\g(0)$. Then $\Delta=\Delta(-1)\sqcup \Delta(0)\sqcup \Delta(1)$ and 
$\Delta(-1)=-\Delta(1)$. It is clear that $\Delta(1)$ contains the highest
root for a suitable choice of $\Delta^+$. Hence $\g(1)\cap\omin\ne\varnothing$ and
$\g(-1)\cap\omin\ne\varnothing$. 

2) The list of symmetric pairs $(\g,\g_0)$ such that $\omin\cap\g_1=\varnothing$ occurs
in~\cite[Remark\,2.2]{bk95}, and the equivalence of {\sf (i)} and {\sf (ii)} in Theorem~\ref{thm:g1-empty}
is also mentioned therein. However, it seems that the proofs promised in that remark have never been 
published.
\end{rmk}

It follows from Vinberg's lemma~\cite[\S\,3, n.2]{vi76} that, for a $G$-orbit $\co\subset\g$, each 
irreducible component of $\co\cap\g_1$ is a $G_0$-orbit. For $\co=\omin$, this can be made more 
precise. We use  the notation of Remark~\ref{rmk:six-ss-g_0}(1).

\begin{prop}         \label{prop:intersect-g1}
If $\omin\cap\g_1\ne\varnothing$, then 

$\omin\cap\g_1=\begin{cases}
\omin(\g_1), & \text{if\/ $\g_1$ is a simple $\g_0$-module,}\\
\omin(\g(1))\sqcup\omin(\g(-1)), & \text{if\/ $\g_1=\g(-1)\oplus\g(1)$.}  \end{cases}$
\end{prop}
\begin{proof}
Recall that $\g_1$ is a simple $\g_0$-module if and only if $\g_0$ is semisimple.

If the closure of an irreducible component of $\omin\cap\g_1$ contains a nontrivial $G_0$-orbit $\co_0$,
then $\dim\co_0 < \dim(\omin\cap\g_1)=(\dim\omin)/2$. Then $G{\cdot}\co_0$ is a $G$-orbit
such that $\dim G{\cdot}\co_0= 2\dim\co_0< \dim\omin$. A contradiction! Hence each irreducible 
component of $\omin\cap\g_1$ is the minimal $G_0$-orbit in a simple $\g_0$-submodule of $\g_1$.
\end{proof}

Example~\ref{ex:sln-max}(2) provides an illustration to the second possibility of 
Proposition~\ref{prop:intersect-g1}. In that example, $\g_0=\gln$, 
$\g_1=\mathsf{Sym}(n)\oplus\mathsf{Sym}(n)^*$ as $\gln$-module, and $\omin\cap\g_1=
\mathsf{Sym}_1(n)\sqcup\mathsf{Sym}_1(n)$.

\section{Projections of $\omin$ to $\g_0$ and $\g_1$}
\label{sect:projections}

\noindent
In this section,  the varieties $\ov{\psi(\omin)}\subset\g_1$ and $\ov{\vp(\omin)}\subset\g_0$ are 
described. Both of them are irreducible, conical, and $G_0$-stable. We show that their structure 
essentially depends on whether the intersection $\omin\cap\g_1$ is empty or not. 
\\ \indent 
Unless otherwise stated, $\{e,h,f\}$ is a minimal $\tri$-triple.
In the presence of a $\BZ_2$-grading, an $\tri$-triple is said to be {\it normal}, if $e,f\in\g_1$ 
and $h\in\g_0$. By~\cite[Prop.\,4]{kr71}, every $e\in\N\cap\g_1$ can be included in a normal $\tri$-triple.
Recall from Section~\ref{subs:secant} that 
\beq    \label{eq:sec-Omin}
    \CS(\omin)=\eus K(G{\cdot}h)\subset \g .
\eeq 
The following result provides a link between both projections and conical secant varieties.
\begin{lm}    \label{lm:secant-sigma}
For any $\sigma\in\mathsf{Inv}(\g)$, one has 
\\[.6ex] 
\centerline{ 
$\psi(\omin)\subset \CS(\omin)\cap\g_1$ and $\vp(\omin)\subset \CS(\omin)\cap\g_0$.}
\end{lm}
\begin{proof}
Let $x=x_0+x_1\in \omin$ with $x_i\in\g_i$. Then $\sigma(x)=x_0-x_1\in \sigma(\omin)=\omin$. 
For a generic $x$, both summands are nonzero. 
Hence $[x_0], [x_1]\in \Sec(\omin)$. Therefore
$x_0=\vp(x)\in \CS(\omin)\cap\g_0$ and $x_1=\psi(x)\in \CS(\omin)\cap\g_1$.
\end{proof}

\subsection{The case with $\omin\cap\g_1\ne\varnothing$}  
\label{subs:proj-g1}
There are two reasons why the intersection $\omin\cap\g_1$ is important. {\it First}, one has
$\omin\cap\g_1\ne\varnothing \ \Longleftrightarrow \ G{\cdot}h \cap\g_1\ne\varnothing$, which is a special case of~\cite[Theorem\,1]{leva}, cf. Remark~\ref{rem:Levon}. 
The implication ``$\Rightarrow$'' is easy (and will be presented in the next proof), while the converse is 
not. {\it Second}, $\omin\cap\g_1\ne\varnothing$ if and only if normal minimal $\tri$-triples exist.

\subsubsection{Projections to $\g_1$}
\begin{thm}           \label{thm:psi-omin-gen}
Suppose that $\omin\cap\g_1\ne\varnothing$. Let $\{e,h,f\}$ be a normal minimal $\tri$-triple and
$h_1=e-f \in\g_1$. Then
\begin{itemize}
\item[\sf (i)] \  $h_1 \in \eus K(G{\cdot}h)\cap\g_1$; 
\item[\sf (ii)] \ $\ov{\psi(\omin)}=\eus K(G_0{\cdot}h_1)$ \ and \ $\dim \psi(\omin)=\dim\omin$.
\end{itemize}
\end{thm}
\begin{proof}
{\sf (i)} If $i=\sqrt{-1}$, then $ih_1=\left(\begin{smallmatrix} 0 & i \\ -i& 0 \end{smallmatrix}\right)$ is conjugate to 
$h=\left(\begin{smallmatrix} 1 & 0 \\  0& -1\end{smallmatrix}\right)$ in $\lg e,h,f\rg\simeq \tri$ and
thereby in $\g$. Hence $i h_1\in G{\cdot}h\cap\g_1$. 

\noindent
{\sf (ii)} \ Since $h$ is semisimple, $G{\cdot}h\cap\g_1$ is a sole
$G_0$-orbit~\cite[Prop.\,6.6]{bul}, hence $G{\cdot}h\cap\g_1=G_0{\cdot}ih_1$.  
Using Proposition~\ref{prop:kr71} and~\eqref{eq:dim-Gh}, we obtain
\[
  \dim G_0{\cdot}h_1=\dim G_0{\cdot}ih_1=\frac{1}{2}\dim G{\cdot}h= \dim\omin-1 .
\]
Then $\dim\eus K(G_0{\cdot}h_1)=\dim G_0{\cdot}h_1+1=\dim\omin$. 
Next step is to show that $h_1\in \psi(\omin)$. Indeed, one has
\beq    \label{eq:exp(ad-f)}
   v:=\exp(\ad f){\cdot}e=e+[f,e]+\frac{1}{2}[f,[f,e]]=e-h-f\in\omin
\eeq
and $\psi(v)=e-f=h_1$. Since $\ov{\psi(\omin)}$ is a cone,
the whole variety $\eus K(G_0{\cdot}h_1)$ is contained in 
$\ov{\psi(\omin)}$. Because the latter is irreducible and $\dim\ov{\psi(\omin)}\le
\dim\omin=\dim\eus K(G_0{\cdot}h_1)$, the result follows.
\end{proof}

Since $\dim\omin=\dim\psi(\omin)$ in Theorem~\ref{thm:psi-omin-gen}, the degree of $\psi$,
$\deg(\psi)$, is well-defined. By definition, $\deg(\psi)=\#\psi^{-1}(z)$ for generic 
$z\in\ov{\psi(\omin)}$. Here the structure of $\ov{\psi(\omin)}$ shows that $\deg(\psi)=\#\psi^{-1}(h_1)$.

\begin{thm}      \label{thm:deg-psi=2}
If $\omin\cap\g_1\ne\varnothing$, $\{e,h,f\}$ is a normal $\tri$-triple for $\omin$, and $h_1=e-f$, then
\[
          \psi^{-1}(h_1)=\{e-h-f, e+h-f\}=\{h_1-h, h_1+h\} .
\]
In particular, $\deg(\psi)=2$. 
\end{thm}
\begin{proof}
Computations in $\lg e,h,f\rg\simeq\tri$ show that $\{e-h-f, e+h-f\}\subset \psi^{-1}(h_1)=\psi^{-1}(e-f)$.
Conversely, suppose that $z:=h_1+x\in \psi^{-1}(h_1)$ for some $x\in\g_0$. Then
$\Ima(\ad z)^2=\lg z\rg$, see~\eqref{eq:ohne-Z-grad}. 
In particular, $(\ad z)^2(e)=\zeta z$ and $(\ad z)^2(f)=\eta z$ for some
$\zeta,\eta\in\bbk$.   \par
Equating the $\g_0$- and $\g_1$-components in both relations yields 
two systems of equations:
\[
   \begin{cases}   [f,[e,x]]-[h,x]=\zeta x  \\
                          (\ad x)^2 e-2(e+f)=\zeta h_1=\zeta e-\zeta f \end{cases} \ \text{ and } \ 
    \begin{cases}   [e,[x,f]]-[h,x]=\eta x  \\
                          (\ad x)^2 f-2(e+f)=\eta h_1=\eta e-\eta f  .\end{cases}                      
\]
For computations with the $\g_0$-components, we use the fact that $[e,[x,e]]\in \g_0\cap\lg e\rg=\{0\}$, 
and likewise for $f$. Then we rearrange the above relations as follows:
\[
   ({\rm I}): \begin{cases}   [f,[e,x]]-[h,x]=\zeta x  \\
                          -[e,[f,x]]-[h,x]=\eta x  \end{cases} \ \text{ and } \ 
    ({\rm II}): \begin{cases} (\ad x)^2 e=(\zeta+2)e+(2-\zeta)f    \\
                           (\ad x)^2 f= (\eta+2)e+(2-\eta)f  .\end{cases}                      
\]
\textbullet \quad System (II) means that the plane $\lg e,f\rg$ is $(\ad x)^2$-stable and
\[
    ({\rm II}'): \begin{cases} (\ad x)^2 (e-f)=(\zeta-\eta)(e-f)    \\
                           (\ad x)^2 (e+f)= (\zeta+\eta)(e-f) +4(e+f) .\end{cases} 
\]
\textbullet \quad Taking the sum in (I) and using the Jacobi identity, we obtain $-3[h,x]=(\zeta+\eta)x$.
\\
\textbullet \quad Assume that $\zeta+\eta\ne 0$. Then $x$ is nilpotent and hence the determinant of 
(II) or (II)' must be zero, i.e., $\zeta=\eta$. But in this case we have 
$(\ad x)^2 (e-f)=0$ and $(\ad x)^2 (e+f)= 2\zeta(e-f) +4(e+f)$. Therefore, $(\ad x)^{2k}(e+f)\ne 0$ for all
$k\in\BN$, i.e., $\ad x$ is not nilpotent. This contradiction means that $\zeta+\eta=0$ and $[h,x]=0$.
\\
\textbullet \quad For the minimal $\tri$-triples, one has $\g^h=\lg h\rg\dotplus\el$, where $\el$
is the centraliser of the whole $\tri$-triple $\{e,h,f\}$. Therefore, 
\[
    x=ah + \tilde h,
\]
where $\tilde h\in \el$ and $a\in\bbk$. Substituting this $x$ into system (I) and using the relations
$[\tilde h,e]=[\tilde h,f]=0$, one readily obtains $\tilde h=0$, 
$\zeta=2$, and $\eta=-2$.
\\
\textbullet \quad Thus, $z=h_1+x=e-f+ah \in \omin$. But elements of this form are nilpotent in
$\lg e,h,f\rg\simeq \tri$ if and only if $a=\pm 1$.
\end{proof}

By Theorem~\ref{thm:g0-empty}, there are two cases in which $\psi$ is finite (of degree $2$). A more
precise assertion for them is

\begin{prop}          \label{prop:faktor-C2-psi}
If $\omin\cap\g_0=\varnothing$, then $\psi:\bomin\to \psi(\bomin)$ is the categorical quotient by the linear 
action of the cyclic group\/ $\mathcal C_2$ with generator $(-\sigma)\in GL(\g)$. 
\end{prop}
\begin{proof}
Let $\pi: \bomin \to \bomin/\mathcal C_2$ be the quotient morphism. Since the map 
\[ (x\in\bomin) \mapsto (\psi(x)=(x-\sigma(x))/2\in \g_1) \] 
is $\mathcal C_2$-equivariant, we obtain the commutative diagram
\[
\xymatrix{    \bomin  \ar[d]^{\pi}  \ar@/^/ [dr]^{\psi} & \\
                      \bomin/\mathcal C_2 \ar[r]^\kappa        &  \psi(\bomin)\ , }        
\]  
where $\psi$ is onto. Our goal is to prove that $\kappa$ is an isomorphism.
Since $\deg(\psi)=\deg(\pi)=2$, the map $\kappa$ is birational. Since $\bomin$ is 
normal~\cite{vp72}, so is $\bomin/\mathcal C_2$. Therefore, it suffices to prove that
$\psi(\bomin)$ is a normal variety. And this follows from the description of $\Ima(\psi)$ in
Example~\ref{ex:sln-max}.
\end{proof}
\subsubsection{Projections to $\g_0$}
To describe $\ov{\vp(\omin)}$, we need Lemma~\ref{lm:secant-sigma} and the relation
$\eus K(G{\cdot}h)=\CS(\omin)$. 
Since $G{\cdot}h$ is a closed (hence non-conical) $G$-orbit, it follows from~\eqref{eq:dim-Gh} that
\[
   \dim\eus K(G{\cdot}h)=\dim G{\cdot}h +1=2\dim\omin-1 .
\]
\begin{thm}           \label{thm:vp-omin-gen}
Suppose that $\omin\cap\g_1\ne\varnothing$, and let $\{e,h,f\}$ be a normal minimal $\tri$-triple.
Then $\ov{\vp(\omin)}=\eus K(G_0{\cdot}h)$ and $\dim\vp(\omin)=\dim\omin-1$.
\end{thm}
\begin{proof}
By~\eqref{eq:exp(ad-f)}, we have $e-h-f\in\omin$. Hence $-h\in\vp(\omin)$ and 
$\eus K(G_0{\cdot}h)\subset\ov{\vp(\omin)}$. On the other hand,
$\ov{\vp(\omin)}\subset \eus K(G{\cdot}h)\cap\g_0$ (Lemma~\ref{lm:secant-sigma}), and 
since $G_0{\cdot}h$ is an irreducible component of $G{\cdot}h \cap\g_0$~\cite[Theorem\,3.1]{R67}, 
it is clear that $\eus K(G_0{\cdot}h)$ is an irreducible component of $\eus K(G{\cdot}h)\cap\g_0$.
Hence $\ov{\vp(\omin)}=\eus K(G_0{\cdot}h)$.

To compute $\dim G_0{\cdot}h$, we use the $(\BZ,h)$-grading related to our normal minimal 
$\tri$-triple, see~\eqref{eq:Z-h-grad}. By the assumption, we have $\g(2)\oplus\g(-2)=\lg e,f\rg \subset\g_1$. Hence
$\g_0\subset \g(1)\oplus\g(0)\oplus\g(-1)$. Furthermore, if $v\in \g(1)\cap \g_0$, then
$[f,v]\subset \g(-1)\cap\g_1$, and vice versa. (Note that $[e,[f,v]]=v$ for any $v\in\g(1)$.) This implies that
$\dim \Bigl(\g_0\cap \bigl(\g(1)\oplus\g(-1)\bigr)\Bigr)=\dim\g(1)$. Therefore,
\[
    \dim G_0{\cdot}h=\dim\g_0-\dim(\g_0\cap\g(0))=\dim\g(1)=\dim \omin-2 .
\]
Thus, $\dim \eus K(G_0{\cdot}h)=\dim \omin-1$, and we are done.
\end{proof}

Since $\dim\omin-\dim\vp(\omin)=1$, generic fibres of $\vp:\bomin\to \eus K(G_0{\cdot}h)$ are one-dimensional. Therefore, $\dim \vp^{-1}(h)=1$.
Here a slight modification of~\eqref{eq:exp(ad-f)} provides a clue about the structure of $\vp^{-1}(h)$.
For any $t\in\bbk^*$, one has 
\beq    \label{eq:exp(tad-f)}   
  \exp(\ad {t}e){\cdot}\frac{1}{t}f=\frac{1}{t}f+[e,f]+\frac{1}{2}[te,[te,\frac{1}{t}f]]=\frac{1}{t}f+h-te \in
  \vp^{-1}(h) .
\eeq
We prove below that this actually yields the whole fibre $\vp^{-1}(h)$ and thereby generic fibres of $\vp$ 
are irreducible.
\begin{thm}      \label{thm:fibre-phi}
Suppose that $\omin\cap\g_1\ne\varnothing$. Let $\{e,h,f\}$ be a normal minimal $\tri$-triple and
$\vp:\bomin\to \ov{\vp(\omin)}=\eus K(G_0{\cdot}h)$. Then 
$\vp^{-1}(h)=\{\frac{1}{t}f+h-te \in \vp^{-1}(h)\mid t\in \bbk^*\}$.
\end{thm}
\begin{proof}
The inclusion `$\supset$' is shown in~\eqref{eq:exp(tad-f)}. 

Our proof for `$\subset$' is close in spirit to the proof of Theorem~\ref{thm:deg-psi=2}, though some 
details are different. Suppose that $z=h+x\in \vp^{-1}(h)$ for some $x\in\g_1$. Then 
$\Ima(\ad z)^2=\lg z\rg$, see~\eqref{eq:ohne-Z-grad}. In particular, $(\ad z)^2(e)=\zeta z$ and 
$(\ad z)^2(f)=\eta z$ for some $\zeta,\eta\in\bbk$.   
\par
Equating the $\g_0$- and $\g_1$-components in both relations yields two systems of equations:
\[
   \begin{cases}   [h,[x,e]]+2[x,e]=\zeta h  \\
                          (\ad x)^2 e+4e=\zeta x \end{cases} \ \text{ and } \ 
    \begin{cases}   [h,[x,f]]-2[x,f]=\eta h  \\
                          (\ad x)^2 f + 4f=\eta x .\end{cases}                      
\]
Using the Jacoby identity and certain rearrangements,  we obtain:
\[
   ({\rm I}): \begin{cases}   [[h,x]+4x,e]=\zeta h  \\
                          [[h,x]-4x,f]=\eta h   \end{cases} \ \text{ and } \ 
    ({\rm II}): \begin{cases} (\ad x)^2 e=-4e+\zeta x    \\
                           (\ad x)^2 f= -4f+\eta x  .\end{cases}                      
\]
For a while, we concentrate on system (I). Using the relations for the $\tri$-triple, one derives from
(I) that
\[
   ({\rm I'}): \begin{cases}   [h,x]+4x=-\zeta f +a_e  \\
                          [h,x]-4x=\eta e+b_f ,  \end{cases}
\]
where $a_e\in \g^e$ and $b_f\in\g^f$. Taking the sum and difference in (I'), we obtain
\begin{gather}
\label{eq:summa}  2[h,x] = \eta e-\zeta f +a_e+b_f ,\\ 
\label{eq:raznost}   8x=  -\eta e-\zeta f +a_e-b_f .
\end{gather}
Next, we apply $\ad h$ to equality~\eqref{eq:raznost} and then substitute $[h,x]$ 
from~\eqref{eq:summa} into the equality obtained. We get
\beq           \label{eq:star}
          6\eta e-6\zeta f=-4a_e-4b_f +[h,a_e]-[h,b_f] .
\eeq
Recall that associated with the (normal) $\tri$-triple $\{e,h,f\}$, we have the $(\BZ,h)$-grading
$\g=\bigoplus_{i=-2}^2\g(i)$, see~\eqref{eq:Z-h-grad}.
Here $a_e\in \g^e\subset \g(\ge 0)$ and $b_f\in\g^f\subset\g(\le 0)$, i.e.,
$a_e=a_0+a_1+a_2$ and $b_f=b_0+b_{-1}+b_{-2}$, where $a_i\in\g(i)$ and $b_{-j}\in \g(-j)$.

\textbullet\quad Comparing the $\BZ$-homogeneous components of~\eqref{eq:star} in
$\g(i)$ with $i=-1,0,1$, we conclude that
$a_1=b_{-1}=0$ and $a_0=-b_0=:2r$. This also means that $r\in\z_\g(e,h,f)$.

\textbullet\quad Since $\g(2)=\lg e\rg$, we have $a_2=ce$ and then comparing the $\g(2)$-components 
in~\eqref{eq:star} implies that
$c=-3\eta$. Likewise, $b_{-2}=\tilde c f$ and then $\tilde c=3\zeta$. 

Thus, our current achievement is that $a_e=a_0+ce=
2r-3\eta e$ and $b_f=b_0+\tilde cf=-2r+3\zeta f$. Substituting these $a_e$ and $b_f$ into~\eqref{eq:raznost}, we obtain 
\[
      x=\frac{1}{2}(-\eta e-\zeta f +r) .
\]
Let us remember that there are also equations~(II). Substituting this $x$ into, say, the 
first equation in (II), we obtain
\[
     \zeta\eta e=-4e +\zeta r/2 .
\]
Therefore, $\zeta\eta=-4$  and $r=0$. In other words, $x=-te+\frac{1}{t}f$, as required.
\end{proof}

\subsection{The case with $\omin\cap\g_1=\varnothing$} 
\label{subs:empty}
If $\omin\cap\g_1=\varnothing$, then $G{\cdot}h\cap\g_1=\varnothing$ and normal minimal $\tri$-triples
do not exist. On the other hand, we 
have assertions of Theorem~\ref{thm:g1-empty} at our disposal, and there are only six such cases (Example~\ref{ex:6-cases}). 
\subsubsection{Projections to $\g_1$}
\label{ssubs:pusto-g1}
We give a uniform description of $\ov{\psi(\omin)}$ using (partly) a case-by-case argument. 

\begin{thm}         \label{thm:psi-omin-6}  
Suppose that $\omin\cap\g_1=\varnothing$.  Then  
\begin{itemize}
\item[\sf (i)] \ $\ov{\psi(\omin)}$ is the closure of a nilpotent $G_0$-orbit in $\g_1$ and 
$\dim\psi(\omin)<\dim\omin$. 
\item[\sf (ii)] \ More precisely, $\ov{\psi(\omin)}=\ov{\omin(\g_1)}$.
\end{itemize}
\end{thm}
\begin{proof}
{\sf (i)} By Theorem~\ref{thm:g1-empty}, $G_0$ has a dense orbit in $\omin$. Then $\ov{\psi(\omin)}$
has a dense $G_0$-orbit, too. Since $\ov{\psi(\omin)}$ is conical, this $G_0$-orbit is also conical, hence
nilpotent. 

To estimate $\dim\psi(\omin)$, we use $\CS(\omin)$. By~\eqref{eq:sec-Omin} and 
Lemma~\ref{lm:secant-sigma}, we have $\ov{\psi(\omin)}\subset \eus K(G{\cdot}h)\cap\g_1$. 
By~\eqref{eq:dim-Gh},
all $G$-orbits in $\eus K(G{\cdot}h)$ has dimension at most $2\dim\omin-2$. On the other hand, for
any $G$-orbit $\hat\co\subset\g$, each irreducible component of $\hat\co\cap\g_1$ is a 
$G_0$-orbit~\cite[\S\,3, n.2]{vi76} and $\dim(\hat\co\cap\g_1)=\frac{1}{2}\dim\hat\co$. 
Therefore, it follows from the preceding paragraph that $\ov{\psi(\omin)}$ is the closure of an irreducible 
component of $\hat\co\cap\g_1$ for some $\hat\co\subset \eus K(G{\cdot}h)\cap\N$
and hence $\dim\psi(\omin)\le\dim\omin-1$.

{\sf (ii)} Since $\omin\cap\g_1=\varnothing$, the subalgebra $\g_0$ is semisimple and $\g_1$ is a simple 
$G_0$-module, see Remark~\ref{rmk:six-ss-g_0}(1). Hence there is the minimal 
$G_0$-orbit  $\omin(\g_1)$ and $\omin(\g_1)\subset \ov{\psi(\omin)}$. Then 
$\tilde\co:=G{\cdot}\omin(\g_1)$ is a nilpotent $G$-orbit in $\eus K(G{\cdot}h)$, and 
$\tilde\co\ne \omin$ by the 
assumption. The $G$-orbits in $\eus K(G{\cdot}h)$ are known~\cite{ky00}, and 
using $\sat(\sigma)$ and Remark~\ref{rem:Levon}, one readily verifies that 
$\tilde\co$ is the only $G$-orbit in $\eus K(G{\cdot}h)$ meeting $\g_1$, cf. examples below.  Hence  
$\tilde\co\cap\g_1$ is irreducible and $\ov{\psi(\omin)}=\ov{\tilde\co\cap\g_1}=\ov{\omin(\g_1)}$. 
\end{proof}

\begin{ex}   \label{ex:2-bad-except-data}
We point out $\sat(\sigma)$ and the weighted Dynkin diagrams of the nilpotent orbits in $\CS(\omin)$
for two symmetric pairs in Example~\ref{ex:6-cases} with exceptional $\g$. We use the standard labels
for the nilpotent orbits in the exceptional Lie algebras, see~\cite[Ch.\,8]{CM}.

\vspace{1ex}

\centerline{
\begin{tabular}{c}
$(\GR{E}{6},\GR{F}{4})$ \\ 
\begin{picture}(80,25)(35,10) 
\setlength{\unitlength}{0.016in} 
\multiput(38,18)(15,0){4}{\line(1,0){9}}
\put(65,3){\circle*{5}}
\multiput(50,18)(15,0){3}{\circle*{5}}
\multiput(35,18)(60,0){2}{\circle{5}}
\put(65,6){\line(0,1){9}}
\end{picture}
\end{tabular}   
\quad  
\begin{tabular}{r|ccc||ccc|}
& $\co$ &  \quad $\eus D(\co)$ & $\dim\co$ & $\co$ &  \quad $\eus D(\co)$ & $\dim\co$ \\ \hline\hline
$\omin$: & $\eus A_1$ & \quad\begin{E6}{0}{0}{0}{0}{0}{1}\end{E6}\quad & 22 & $3\eus A_1$ & \quad\begin{E6}{0}{0}{1}{0}{0}{0}\end{E6} & 40 \\
$\tilde\co$: & $2\eus A_1$ & \quad\begin{E6}{1}{0}{0}{0}{1}{0}\end{E6} & 32 &
$\eus A_2$ & \quad\begin{E6}{0}{0}{0}{0}{0}{2}\end{E6} & 42 \\ \hline
\end{tabular}
}

\vspace{2ex}

\centerline{
\begin{tabular}{c}
$(\GR{F}{4},\GR{B}{4})$ \\ 
\begin{picture}(65,15)(60,4)  
\setlength{\unitlength}{0.016in} 
\put(50,8){\circle{5}}
\multiput(68,8)(18,0){3}{\circle*{5}}
\multiput(53,8)(36,0){2}{\line(1,0){13}}
\multiput(70.5,6.6)(0,2.6){2}{\line(1,0){13.8}}
\put(73,5.2){$<$} 
\end{picture}
\end{tabular}   \quad 
\begin{tabular}{r|ccc||ccc|}
& $\co$ &  \ $\eus D(\co)$ & $\dim\co$ & $\co$ &  \ $\eus D(\co)$ & $\dim\co$ \\ \hline\hline
$\omin$:  & $\eus A_1$ & \  0--0--0--1 \  & 16 & $\eus A_1{+}\tilde{\eus A}_1$ & \  0--0--1--0 \  & 28 \\
$\tilde\co$:  & $\tilde{\eus A}_1$ & \  1--0--0--0 \  & 22 & $\eus A_2$ & \  0--0--0--2 \  & 30 \\ \hline
\end{tabular}
}
\vspace{1ex}

\noindent
In both cases, $\omin=\eus A_1$ and the next orbit is $\tilde\co$. Looking at these data,
one sees that $\tilde\co$ is the only orbit such that the set of black nodes of $\sat(\sigma)$ is contained 
in the set of zeros of the weighted Dynkin diagram,  
i.e., $\tilde\co\cap\g_1\ne\varnothing$, see Remark~\ref{rem:Levon}.
\end{ex}

\begin{ex}       \label{ex:1-bad-class-data}
1) \ For $(\sltn,\spn)$, $n\ge 2$, the Satake diagram is:
 \begin{picture}(115,20)(5,5)
\put(10,8){\circle*{6}}    
\put(30,8){\circle{6}}      
\put(13,8){\line(1,0){14}}
\multiput(70,8)(40,0){2}{\circle*{6}} 
\put(90,8){\circle{6}}
\multiput(73,8)(20,0){2}{\line(1,0){14}}
\multiput(33,8)(29,0){2}{\line(1,0){5}}  
\put(42,5){$\cdots$}
\end{picture} 
\\ ($2n{-}1$ nodes) and the orbits in $\eus K(G{\cdot}h)\cap\N$ are:  \vspace{1ex}

\centerline{
\begin{tabular}{r|ccc|}
 & $\blb(\co)$ &  \ $\eus D(\co)$ & $\dim\co$  \\ \hline \hline
$\omin$:     & $(2,1,\dots,1)$   & \  1--0--$\cdots$--0--1 \         & $4n{-}2$  \\
$\tilde\co$: & $(2,2,1,\dots,1)$ & \  0--1--0-$\cdots$-0--1--0 \  & $8n{-}8$ \\ 
                  &  $(3,1,\dots,1)$ & \  2--0--$\cdots$--0--2 \  & $8n{-}6$ \\ \hline  
\end{tabular} }  

\vspace{.8ex}

2) For $(\spn,\mathfrak{sp}_{2k}\dotplus\mathfrak{sp}_{2n-2k})$, $k\le n-k$, the Satake diagram is:

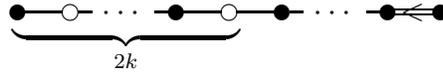
\centerline{\begin{picture}(165,25)(0,-5)
\multiput(30,8)(60,0){2}{\circle{6}}
\multiput(70,8)(40,0){3}{\circle*{6}}
\multiput(10,8)(160,0){2}{\circle*{6}}
\multiput(152.5,7)(0,2){2}{\line(1,0){15}}
\multiput(73,8)(20,0){2}{\line(1,0){14}}
\multiput(33.1,8)(29,0){2}{\line(1,0){5}}
\multiput(113,8)(29,0){2}{\line(1,0){5}}
\put(42,5){$\cdots$}
\put(122,5){$\cdots$}
\put(13,8){\line(1,0){14}}
\put(8,2){$\underbrace%
{\mbox{\hspace{86\unitlength}}}_{2k}$}
\put(155,5){$<$} 
\end{picture} 
}  
and the orbits in $\eus K(G{\cdot}h)\cap\N$ are:  
\vspace{1ex}

\centerline{
\begin{tabular}{r|ccc|}
 & $\blb(\co)$ &  \ $\eus D(\co)$ & $\dim\co$  \\ \hline \hline
$\omin$:     & $(2,1,\dots,1)$   & \  1--0--$\cdots$--0 \         & $2n$  \\
$\tilde\co$: & $(2,2,1,\dots,1)$ & \  0--1--0-$\cdots$-0 \  & $4n{-}2$ \\ 
                   \hline 
\end{tabular} }  

\noindent
Here $\blb(\co)$ is the partition of $2n$ corresponding to $\co$, and again $\tilde\co$ is the only orbit 
meeting $\g_1$. A similar situation occurs for the involutions of $\son$. The complete list of orbits 
$\tilde\co$ appears in Table~\ref{table:odin}.
\end{ex}

\subsubsection{Projections to $\g_0$}
\label{proj-g_0-empty}
If $\omin\cap\g_1=\varnothing$, then $G_0$ has a dense orbit in $\omin$ and 
$\vp: \bomin\to \ov{\vp(\omin)}=\vp(\bomin)$ is finite, see Theorem~\ref{thm:g1-empty}. Therefore, 
$\vp(\bomin)$ is the closure of a nilpotent $G_0$-orbit in $\g_0$, and we write $\co_0$ for this 
$G_0$-orbit. We prove below that $\hot(G{\cdot}\co_0)=2$, i.e., $(\ad x)^3=0$ for $x\in\co_0$
(Theorem~\ref{thm:height-2}{\sf (i)}). In particular the height of the $G_0$-orbit $\co_0$ equals $2$, too.
If $x\in \vp(\omin)\setminus\omin$ and $y\in\vp^{-1}(x)$, then 
$\sigma(y)\in \vp^{-1}(x)$ and also $y\ne \sigma(y)$. This shows that $\deg(\vp)$
is even. We prove in Theorem~\ref{thm:height-2}{\sf (ii)} below that actually 
$\deg(\vp)=2$ for all six cases with $\omin\cap\g_1=\varnothing$. 
This yields a $\vp$-analogue of Proposition~\ref{prop:faktor-C2-psi}.

\begin{prop}             \label{prop:faktor-C2-fi}
If $\omin\cap\g_1=\varnothing$, then $\vp:\bomin\to \ov{\co_0}$ is the categorical quotient by the linear 
action of the cyclic group\/ $\hat{\mathcal C}_2$ with generator $\sigma\in GL(\g)$.
\end{prop}
\begin{proof}
Let $\pi: \bomin \to \bomin/\hat{\mathcal C}_2$ be the quotient morphism. Since the map 
\[ (x\in\bomin) \mapsto (\vp(x)=(x+\sigma(x))/2\in\g_0) \] 
is $\hat{\mathcal C}_2$-equivariant, we obtain the commutative diagram
\[
\xymatrix{    \bomin  \ar[d]^{\pi}  \ar@/^/ [dr]^{\vp} & \\
                      \bomin/\hat{\mathcal C}_2 \ar[r]^\kappa        &  \ov{\co_0}\ , }        
\]  
where $\vp$ is onto. Since $\deg(\vp)=\deg(\pi)=2$, the map $\kappa$ is birational.  Again, to prove that 
$\kappa$ is an isomorphism, it suffices to know that $\ov{\co_0}$ is a normal variety. Here
$\hot(\co_0)=2$, and 
by an old result of Hesselink~\cite[Sect.\,5]{he79}, the nilpotent orbits of height $2$ have normal closure.
\end{proof}

In Table~\ref{table:odin}, we gather information on $\ov{\psi(\omin)}$ and $\ov{\vp(\omin)}$ for the six 
cases in which $\omin\cap\g_1=\varnothing$. If $\g$ or $\g_0$ is a classical Lie algebra, then the 
corresponding orbits are represented by partitions. For $\GR{C}{k}$, the partition of $\omin$ is 
$(2,1^{2k-2})$. As in Section~\ref{ssubs:pusto-g1}, $\tilde\co$ is the 
nilpotent $G$-orbit such that $\tilde\co\cap\g_1=\ov{\omin(\g_1)}=\ov{\psi(\omin)}$ and hence 
$\dim\psi(\omin)=\frac{1}{2}\dim\tilde\co$. It is the only orbit in $\CS(\omin)$ meeting $\g_1$. While 
$\co_0$ is the nilpotent $G_0$-orbit such that  $\ov{\vp(\omin)}=\ov{\co_0}$. To explicitly identify 
$\co_0\subset\g_0$, we use the relations $\hot(\co_0)=2$ and $\dim\omin=\dim\co_0$ (since $\vp$ is finite). 
For the classical cases, one can also exploit some easy arguments with partitions.

\begin{table}[ht]
\caption{}   \label{table:odin}
\begin{tabular}{c|>{$}c<{$}| >{$}c<{$} >{$}c<{$}| >{$}c<{$} >{$}c<{$}| >{$}c<{$}|} 
{\rus N0} &\g,\g_0 & \dim\g_1 & \dim\omin & \tilde\co & \dim\tilde\co & \co_0  \\ \hline\hline
1&\GR{E}{6},\GR{F}{4} & 26 & 22 & 2\eus{A}_{1} & 32 & \rule{0pt}{2.5ex} \tilde{\eus A}_{1} \\ 
2&\GR{F}{4},\GR{B}{4} & 16 & 16 & \tilde{\eus A}_{1} & 22 & (2^4,1) \\ 
3&\GR{D}{n+1},\GR{B}{n} & 2n{+}1 & 4n{-}2 & (3,1^{2n-1}) & 4n & (3,1^{2n-2}) \\ 
4&\GR{B}{n},\GR{D}{n} & 2n & 4n{-}4 & (3,1^{2n-2}) & 4n{-}2 & (3,1^{2n-3}) \\ 
5&\GR{A}{2n{-}1},\GR{C}{n} & 2n^2{-}n{-}1 & 4n{-}2 & (2^2,1^{2n-4}) & 8n{-}8 & (2^2,1^{2n-4})\\ 
6&\GR{C}{n},\GR{C}{k}{\dotplus}\GR{C}{n-k}  & 4k(n{-}k) & 2n & (2^2,1^{2n-4}) & 4n{-}2 & 
\omin(\GR{C}{k}){\times}\omin(\GR{C}{n{-}k}) \\ 
\hline
\end{tabular}
\end{table}

\subsection{Summary} It follows from our results that, for any $\sigma\in\mathsf{Inv}(\g)$,
exactly one of the two projections $\vp:\omin\to\g_0$ and $\psi:\omin\to\g_1$ is generically finite-to-one.

\textbullet \quad  If $\omin\cap\g_1\ne\varnothing$, then $\dim\psi(\omin)=\dim\omin$ 
(Theorem~\ref{thm:psi-omin-gen}), while $\dim\vp(\omin)=\dim\omin-1$ 
(Theorem~\ref{thm:vp-omin-gen}). In this case either of the projections
contains a $1$-parameter family of closed $G_0$-orbits.

\textbullet \quad  If $\omin\cap\g_1=\varnothing$, then $\dim\psi(\omin)<\dim\omin$ 
(Theorem~\ref{thm:psi-omin-6}), while $\dim\vp(\omin)=\dim\omin$  (Theorem~\ref{thm:g1-empty}). 
In this case either of the projections contains a dense $G_0$-orbit.

\section{Some applications and complements} 
\label{sect:applic}

\subsection{} 
Let $H$ be a connected reductive algebraic subgroup of $G$ with $\h=\Lie H$. Then $\Phi\vert_\h$ is 
non-degenerate and $\g=\h\oplus\h^\perp$ is a direct sum of $H$-modules. Set $\me=\h^\perp$. For
any $x\in\g$, write $x=a+b$, where $a\in\h$ and $b\in\me$.  Let $\vp:\g\to\h$ (resp. $\psi:\g\to\me$)
be the $H$-equivariant projection with kernel $\me$ (resp. $\h$), i.e., $\vp(x)=a$ and $\psi(x)=b$.

\begin{thm}        \label{thm:commut-compon}
Let $\co\subset\g$ be an arbitrary $G$-orbit. Suppose that $H$ has a dense orbit in $\co$, and let $\Omega\subset\co$ be the dense $H$-orbit. Then
\begin{itemize}
\item[\sf (i)] \ the centraliser of\/ $\h$ in $\g$ is trivial. In particular, $\h$ is semisimple; 
\item[\sf (ii)] \ if\/ $x\in\Omega$, then $\h^{a}=\h^x\subset \h^{b}$;
\item[\sf (iii)] \ for any $x=a+b\in\ov{\co}$, we have $[a,b]=0$; 
\item[\sf (iv)] \ $\dim\vp(\co)=\dim\co$;
\item[\sf (v)] \ Suppose that $\co\subset\N$ and $e=a+b\in\co$. Then $a,b\in\N$. \\
Furthermore, if $a+b\in\Omega$ and $h_a\in\h$ is a characteristic of $a$, then $[h_a,b]=2b$. 
\end{itemize}
\end{thm}
\begin{proof}
{\sf (i)} \ Any nonzero $f\in\g^\h$ gives rise to an $H$-invariant polynomial function on $\ov{\co}$. 
 
{\sf (ii)} \ If $x\in\Omega$, then $[\h,x]=[\g,x]$. Hence $\h+\g^x=\g$ and 
$\h^\perp\cap(\g^x)^\perp=\me\cap [\g,x]=\{0\}$. 
If $y\in \h^{a}$, then $[y, x]=[y, b]\subset \me\cap [\g,x]=\{0\}$.

{\sf (iii)} \ Taking $y=a$, we derive from {\sf (i)} that $[a,b]=0$ whenever $x\in\Omega$.
Since the $H$-equivariant polynomial mapping $(x\in\ov{\co})\mapsto ([a,b]\in \me)$ vanishes on 
a dense subset of $\ov{\co}$, it is identically zero. 

{\sf (iv)} \ We have $[\g,x]=\mathbb T_x(\co)$, the tangent space at $x\in\co$. Here
$\textsl{d}\vp_x: \mathbb T_x(\co)\to \mathbb T_{\vp(x)}(\vp(\co))$ has trivial kernel for all $x\in\Omega$.
Therefore $\dim\co=\dim\vp(\co)$.

{\sf (v)} \ Here $\co$ is a conical $G$-orbit. Hence $\Omega$ is a conical $H$-orbit. Therefore,
$\ov{\vp(\co)}=\ov{\vp(\Omega)}$ (resp. $\ov{\psi(\co)}$)
is the closure of a conical $H$-orbit in $\h$ (resp. $\me$).

Consider the $(\BZ, h_a)$-grading of $\g$. Then $[h_a,a]=2a$. Let us write $b$ as a sum of $h_a$-eigenvectors, $b=\sum_{j\in \BZ} b_j$, such that $[h_a,b]=\sum_{j\in \BZ} jb_j$. (Since we know that
$[a,b]=0$, the sum ranges over $j\in\BN$, but this is not important below.)

Let $\Lb:\bbk^*\to G_0$ be the $1$-parameter subgroup corresponding to $h_a$, i.e.,
$\textsl{d}\Lb_1(0)=h_a$. Then $\Lb(t){\cdot}a=t^2a$ and 
$\Lb(t){\cdot}b=\sum_j t^jb_j$.  Since $\co$ is conical, 
\[
       e_t:=t^{-2}\Lb(t){\cdot}e=a+\sum_{j} t^{j-2}b_j\in\co .
\]
This means that $e_t\in \vp^{-1}(\vp(e))=\vp^{-1}(a)$ for any $t\in\bbk^*$. On the other hand, 
if $e=a+b\in\Omega$, then the fibre $\vp^{-1}(a)$
is finite. Clearly, this is only possible if $b=b_2$.
\end{proof}

If $\h=\g_0$ for some $\sigma\in\mathsf{Inv}(\g)$, then $H$ has a dense orbit in $\omin$ if and only if 
$\omin\cap\g_1=\varnothing$. That is, there are six cases in which
Theorem~\ref{thm:commut-compon} applies to $H=G_0$ and
$\co=\omin$.  
\\ \indent
We assume below that $\g$ is simple, $\h=\g_0$, and $\omin\cap\g_1=\varnothing$. For $e\in\omin$, 
we write $e=a+b$, where $a\in\g_0$ and $b\in\g_1$. Then $a\ne 0$ and $a, b\in\N$, see Section~\ref{subs:empty}. 

\begin{thm}    \label{thm:height-2}
Under the assumption that $\h=\g_0$ and $\omin\cap\g_1=\varnothing$, we have 
\begin{itemize}
\item[\sf (i)] \ $(\ad a)^3=(\ad b)^3= 0$ for any $e\in\omin$. In other words,
$\hot(G{\cdot}a)=\hot(G{\cdot}b)=2$;
\item[\sf (ii)] \ If $e=a+b\in\omin$ belongs to the open 
$G_0$-orbit in $\omin$, then $\vp^{-1}(\vp(a))=\{a+b,a-b\}$; hence $\deg(\vp)=2$.
\end{itemize} 
\end{thm}
\begin{proof}
{\sf (i)} \ Take $e\in\omin$ such that $b\ne 0$ (otherwise, there is nothing to prove).
By Theorem~\ref{thm:commut-compon}{\sf (iii)}, we have $[a,b]=0$. Our goal is to prove that 
$(\ad a)^3(y_0)=(\ad a)^3(y_1)=0$ for any $y_i\in\g_i$; and likewise for $b$.     

Let $y=y_0+y_1\in\g$ be arbitrary. Since $e\in\omin$,
\[
    (\ad e)^2(y)=\eta(y) e
\]
for some $\eta=\eta_e\in\g^*$. To simplify notation, we write below $ab(y_i)$ in place of $[[a,[b,y_i]]$,  
$a^2(y_i)$ in place of $[[a,[a,y_i]]$, and likewise for any other ordered word of length ${\ge} 2$ in $a$ 
and $b$. Extracting the components of $(\ad e)^2(y)$ lying in $\g_0$ and $\g_1$, we get
\[
   \begin{cases}  (a^2+b^2)(y_0) + (ab+ba)(y_1)=\eta(y)a \\
                 (ab+ba)(y_0)+ (a^2+b^2)(y_1)=\eta(y)b . \end{cases}
\]
Taking either $y=y_0$ or $y=y_1$, we obtain four relations
\begin{gather*}
{\sf (1)} \ \ (a^2+b^2)(y_0)=\eta(y_0)a, \quad\qquad {\sf (2)} \ \ (ab+ba)(y_1)=\eta(y_1)a, 
\\
{\sf (3)} \ \ (ab+ba)(y_0)=\eta(y_0)b, \quad\qquad {\sf (4)} \ \ (a^2+b^2)(y_1)=\eta(y_1)b. 
\end{gather*}
\indent
\textbullet \quad Set $x_1=a^3(y_0), x_2=ab^2(y_0), x_3=bab(y_0), x_4=b^2a(y_0)$. 
Our goal here is to prove that  $x_1=0$.
Applying $\ad a$ to {\sf (1)} and $\ad b$ to {\sf (3)}, we obtain 
$ x_1+x_2=0$ and  $x_3+x_4=0$. It follows from the relation $[a,b]=0$ that
$x_2=x_3=x_4$. Hence all $x_i=0$.
\\ \indent
\textbullet \quad Set $z_1=a^3(y_1), z_2=ab^2(y_1), z_3=bab(y_1), z_4=b^2a(y_1)$. 
Our next goal is to prove that $z_1=0$. Applying $\ad a$ to {\sf (4)} and $\ad b$ to {\sf (2)}, we obtain 
$ z_1+z_2=0$ and  $z_3+z_4=0$. Again, here $z_2=z_3=z_4$, hence all $z_i=0$.
That is, we have proved that $\hot(G{\cdot}a)=2$.
\\ \indent
\textbullet \quad We leave it to the reader to prove, in a similar way, that $b^3(y_0)=0$ and $b^3(y_1)=0$.

{\sf (ii)} \ Since $\sigma(e)=a-b\in\omin$, we have $a\pm b\in \vp^{-1}(\vp(a))$. Suppose that 
$a+c\in\omin$ for some other $c\in\g_1$. By Theorem~\ref{thm:commut-compon}{\sf (v)}, we have 
$[h_a,c]=2c$, i.e., $c\in\g_1(2)$. Next, a direct verification shows that, for all six cases in which 
$\omin\cap\g_1=\varnothing$, one has $\dim\g_1(2)=1$. Therefore, $c=\eta b$ for some $\eta\in\bbk^*$.
Assume that $\eta\ne \pm 1$. Then $\lg a+b\rg, \lg a-b\rg, \lg a+\eta b\rg$ are three lines in the plane
$\lg a,b\rg$ that belong to $\bomin$. Then the whole plane belongs to $\bomin$ in view of
Lemma~\ref{lm:omin-plane},  and thereby $b\in\omin\cap\g_1$. A contradiction!
\end{proof}

In the notation of Section~\ref{subs:empty}, $G{\cdot}a=G{\cdot}\co_0$ and $G{\cdot}b=\tilde\co$. 
Therefore, Theorem~\ref{thm:height-2}(i) yields an {\sl a priori\/} proof for the fact, which is used in 
Prop.~\ref{prop:faktor-C2-fi}, that the height of the $G_0$-orbit $\co_0$ equals $2$. Moreover, we have 
proved the assertion that $\deg(\vp)=2$, which has also been used in Section~\ref{proj-g_0-empty}. 
We will prove in Section~\ref{subs:rel-between} that $G{\cdot}a=G{\cdot}b$.

\subsection{On the secant variety of $\omin(\g_1)$}
\label{subs:omin-g1}
Suppose that $\g_1$ is a simple $\g_0$-module, i.e., $\g_0$ is semisimple. Then the minimal $G_0$-orbit $\omin(\g_1)$ exists,
and the general description of the conical secant variety $\CS(\omin(\g_1))$ from 
Section~\ref{subs:secant} can be made more explicit. 

Take $\tilde e\in \omin(\g_1)$ and a normal $\tri$-triple $\{\tilde e,h,\tilde f\}$ for the orbit $\tilde \co=G{\cdot}\tilde e\subset\N$.
This yields the $(\BZ,h)$-grading $\g=\bigoplus_{i\in\BZ}\g(i)$. 
We also write $\omin(\g_1)=G_0{\cdot}\tilde e$ and $\BP(\omin(\g_1))=G_0{\cdot}[\tilde e]$.

\begin{thm}          \label{thm:defect-min-g1}
Let\/ $\g_1$ be a simple $\g_0$-module and $\tilde e\in \omin(\g_1)$. Then
\begin{itemize}
\item[\sf (i)] \ $\hot(G{\cdot}\omin(\g_1))=2$ and thereby $\g({\ge}3)=0$;
\item[\sf (ii)] \ $\CS(G_0{\cdot}e)=\eus K(G_0{\cdot}h_1)$, where $h_1=\tilde e+\tilde f$; 
\item[\sf (iii)] \ The defect of\/ $\Sec(G_0{\cdot}[\tilde e])$ equals $\delta_{G_0{\cdot}[\tilde e]}=\dim\g(2)-1$;
\item[\sf (iv)] \ $\delta_{G_0{\cdot}[\tilde e]}=0$ if and only if $\omin(\g_1)=\omin\cap\g_1$.
\end{itemize}
\end{thm}
\begin{proof} 
{\sf (i)} \ If $\omin\cap\g_1\ne\varnothing$, then $G{\cdot}\omin(\g_1)=\tilde\co=\omin$ and we know that 
$\hot(\omin)=2$. \\
If $\omin\cap\g_1=\varnothing$, then 
the assertion follows from Theorem~\ref{thm:height-2}{\sf (i)}.

{\sf (ii)} \ If $\tilde e$ is a highest weight vector in $\g_1$, then $\tilde f$ is a lowest weight vector (under a suitable 
choice of a Cartan and Borel subalgebras of $\g_0$). Then $\CS(G_0{\cdot}\tilde e)=\eus K(G_0{\cdot}h_1)$ 
and $\Sec(G_0{\cdot}[\tilde e])=\ov{G_0{\cdot}[h_1]}$ (Theorem~\ref{thm:sec-min-orb}).

{\sf (iii)} \ To compute the secant defect of $G_0{\cdot}[\tilde e]$,
we use properties of $(\BZ,h)$-gradings and the fact that $h_1\in \bbk^*(G{\cdot}h)$.
Recall that $\g^h=\g(0)$, $\dim\g^{\tilde e}=\dim\g(0)+\dim\g(1)$, and $\dim\g(i)=\dim\g(-i)$. Therefore
\\[.7ex]
\centerline{
$\dim G_0{\cdot}h_1=\frac{1}{2}\dim G{\cdot}h=\dim\g({\ge}1)$ and  
$\dim G_0{\cdot}\tilde e=\frac{1}{2}\dim G{\cdot}\tilde e=\frac{1}{2}\dim\g(1)+ \dim\g({\ge}2)$.}

\noindent
Since $\dim\eus K(G_0{\cdot}h_1)=\dim(G_0{\cdot}h_1)+1$, the defect is equal to 
\[
   \delta_{G_0{\cdot}[\tilde e]}=2\dim G_0{\cdot}\tilde e-\dim\CS(G_0{\cdot}\tilde e)=\dim\g({\ge}2)-1=\dim\g(2)-1 .
\]
\indent {\sf (iv)} \ 
If $\dim\g({\ge}2)=1$, then $\g(2)=\lg \tilde e\rg$. Hence $\tilde e\in \omin$. In this case, we 
also have $\omin(\g_1)=\omin\cap\g_1$ by Prop.~\ref{prop:intersect-g1}.
\end{proof}

\begin{rmk}
It follows from {\sf (iv)} and Example~\ref{ex:6-cases} that there are only six cases in which 
$G{\cdot}\omin(\g_1)\ne \omin$ and hence $\delta_{G_0{\cdot}[\tilde e]}>0$. One can prove that, for 
all involutions $\sigma$ with $\g_0$ semisimple and the $(\BZ,h)$-gradings corresponding to
$G{\cdot}\omin(\g_1)$, one has 
$\dim\g_1(2)=1$. Hence $\delta_{G_0{\cdot}[\tilde e]}=\dim\g(2)-1=\dim\g_0(2)$.
\end{rmk}

\begin{cl}    \label{cor:sec=g1}
$\Sec(\BP(\omin(\g_1))=\BP(\g_1)$ if and only if $G/G_0$ is a symmetric variety of rank $1$.
\end{cl}
\begin{proof}
It follows from the theorem that $\Sec(\BP(\omin(\g_1))=\BP(\g_1)$ if and only if
$\eus K(G_0{\cdot}h_1)=\g_1$, i.e., $\codim (G_0{\cdot}h_1)=1$. Then $r(G/G_0)=1$ in view of 
\eqref{eq:rang-symm-var}.
\end{proof}

As is well known, the symmetric pairs $(\g,\g_0)$ with $\g_0$ semisimple and $r(G/G_0)=1$ are:

\centerline{$(\spn,\mathfrak{sp}_2\dotplus\mathfrak{sp}_{2n-2})$, $n\ge 2$; 
\ $(\son,\mathfrak{so}_{n-1})$, $n\ge 7$; \ $(\mathfrak{sl}_4, \mathfrak{sp}_4)$; \ 
$(\GR{F}{4},\GR{B}{4})$.}

Using Proposition~\ref{prop:intersect-g1} and Theorem~\ref{thm:psi-omin-gen}, we also obtain a connection with projections of $\omin$.

\begin{cl}
If\/ $\omin\cap\g_1\ne\varnothing$, then \ 
$
      \ov{\psi(\omin)}= \CS(\omin\cap\g_1)=\CS(\omin(\g_1))$.
\end{cl}

\subsection{A relation between $\vp(\omin)$ and $\psi(\omin)$}
\label{subs:rel-between}
A careful look at our results in Section~\ref{sect:projections} allows to reveal a curios unexpected
coincidence.  
\begin{thm}            \label{thm:sovpad}
For any $\sigma\in\mathsf{Inv}(\g)$, we have \ 
$\ov{G{\cdot}\vp(\omin)}=\ov{G{\cdot}\psi(\omin)}$. 
More precisely, this common variety is \ \
$\begin{cases}  \CS(\omin), & \text{ if } \ \omin\cap\g_1\ne\varnothing, \\
     \ov{\tilde\co}, & \text{ if } \ \omin\cap\g_1=\varnothing.\end{cases}$ \quad, 
where $\tilde\co=G{\cdot}\omin(\g_1)$.
\end{thm}
\begin{proof}
1$^o$. Suppose that $\omin\cap\g_1\ne\varnothing$.
By Theorems~~\ref{thm:psi-omin-gen} and \ref{thm:vp-omin-gen}, we have 
$\ov{\psi(\omin)}=\eus K(G_0{\cdot}h_1)$ and $\ov{\vp(\omin)}=\eus K(G_0{\cdot}h)$.
Since $h$ and $h_1=e-f$ are $G$-conjugate, the common closure of the $G$-saturation for these two 
varieties is $\CS(\omin)=\eus K(G{\cdot}h)$.

2$^o$. If $\omin\cap\g_1=\varnothing$, then $\ov{G{\cdot}\vp(\omin)}=\ov{G{\cdot}\co_0}$ and
$\ov{G{\cdot}\psi(\omin)}=\ov{\tilde\co}$ (with the notation of Section~\ref{subs:empty}).
Hence the assertion is equivalent to that the $G$-orbits
$G{\cdot}\co_0$ and $\tilde\co$ coincide. Let us check this for the six cases of
Table~\ref{table:odin}. 
\\ \indent
\textbullet \quad
For item~1, one has to check that the $\GR{F}{4}$-orbit of type $\tilde{\eus A}_1$ generates the 
$\GR{E}{6}$-orbit of type $2\eus A_1$. To this end, we use Lawther's tables~\cite{law95}. Let 
$\ka\subset \g_0=\GR{F}{4}$ be the $\tri$-subalgebra corresponding to the orbit $\co_0=\tilde{\eus A}_1$. 
Here $\g_1$ is the $26$-dimensional $\GR{F}{4}$-module.
Then $\g_0\vert_\ka=15\sfr_0+8\sfr_1+7\sfr_2$ (\cite[Table\,4]{law95}) and
$\g_1\vert_\ka=7\sfr_0+8\sfr_1+\sfr_2$ (\cite[Table\,3]{law95}). Hence 
$\g\vert_\ka=22\sfr_0+16\sfr_1+8\sfr_2$. Comparing this with data in \cite[Table\,6]{law95}, we find that $\ka$ yields 
the orbit of type $2\eus A_1$ in $\g=\GR{E}{6}$.
\\ \indent
\textbullet \quad
For item~2, similar computations apply. If $\tri\simeq\ka\subset\mathfrak{so}_9=\mathfrak{so}(\eus V)$ 
corresponds to the orbit $\co_0$ with partition $(2^4,1)$, then $\eus V\vert_\ka=4\sfr_1+\sfr_0$.
Here $\g_0\simeq \bigwedge^2\eus V$ and $\g_1$ is the spinor $\mathfrak{so}_9$-module.
Hence $\g_0\vert_\ka=10\sfr_0+4\sfr_1+6\sfr_2$ 
and $\g_1\vert_\ka=5\sfr_0+4\sfr_1+\sfr_2$. Hence 
$\g\vert_\ka=15\sfr_0+8\sfr_1+7\sfr_2$.
Using again~\cite[Table\,4]{law95}, we conclude that $\g\vert_\ka$ 
corresponds to the orbit of type $\tilde{\eus A}_1$ in $\g=\GR{F}{4}$.
\\ \indent
\textbullet \quad
In item~5, both orbits (in $\sltn$ and $\spn$) have the same partition. An easy argument with partitions
settles also the remaining cases 3,\,4,\,6.
\end{proof}

Another (partial) possibility to proceed in part 2$^o$ of the proof is that $\tilde\co$ and $G{\cdot}\co_0$ are two orbits of height $2$ in $\CS(\omin)$ that are different from $\omin$. And if 
$\g\ne\son$, then there is only one such orbit, i.e., it works for items 1,2,5,6 in Table~\ref{table:odin}.
It might be most illuminating to have a proof of Theorem~\ref{thm:sovpad} that does not resort to
case-by-case considerations.

\begin{rmk}    \label{rmk:conj}
In the same vein, one can prove that $\ov{\vp(\omin)}$ (resp. $\ov{\psi(\omin)}$) is an irreducible 
component of $\ov{G{\cdot}\psi(\omin)}\cap\g_0$  (resp. $\ov{G{\cdot}\vp(\omin)}\cap\g_1$). 
Furthermore,

\textbullet \ \ If $\omin\cap\g_1=\varnothing$, i.e., 
$\ov{G{\cdot}\psi(\omin)}=\ov{G{\cdot}\vp(\omin)}=\ov{\tilde\co}$, then 
$\ov{\tilde\co}\cap\g_1=\ov{\omin(\g_1)}$ is irreducible. Hence $\ov{G{\cdot}\vp(\omin)}\cap\g_1=
\ov{\psi(\omin)}$. However, $\ov{\tilde\co}\cap\g_0$ is reducible for items~2 and 6.

\textbullet \ \ If $\omin\cap\g_1\ne\varnothing$, then $\CS(\omin)\cap\g_0$ can be reducible.
But we do not know whether the intersection $\CS(\omin)\cap\g_1$ is always irreducible.
\end{rmk}

\section{More general projections} 
\label{sect:more}

\noindent
In Section~\ref{sect:projections}, we studied the projections of $\omin$ associated with 
$\sigma\in\mathsf{Inv}(\g)$. Some of the previous results, see e.g. Proposition~\ref{prop:codim-1} and
Theorem~\ref{thm:commut-compon}, suggest that this setup can be generalised in several ways.
\\ \indent
{\it First}, one can eliminate involutions and study projections of $\omin$ associated with arbitrary 
reductive subalgebras $\h\subset\g$ in place of symmetric subalgebras $\g_0$. 
\\ \indent
{\it Second}, it makes sense to consider projections to $\g_0$ or $\g_1$ for
other nilpotent orbits. 

\noindent
Here we only provide a preliminary incomplete treatment for both. A thorough exposition will appear
elsewhere.

\subsection{Projections of $\omin$ associated with reductive subalgebras of $\g$}    
\label{subs:proj-on-h}
Let $G/H$ be an affine homogeneous space. Here $\h=\Lie H$ is reductive and $\Phi\vert_\h$ is 
non-degenerate. As above, $\me=\h^\perp$ is a complementary $\h$-module, i.e., $\g=\h\oplus\me$ and
$[\h,\me]\subset\me$. Let $\vp: \bomin \to \h$ and $\psi: \bomin \to \me$
be the corresponding projections. Clearly, both $\vp$ and $\psi$ are $H$-equivariant.
Using previous results, we obtain the following.

\begin{lm}    \label{lm:empty-m}
Suppose that $\omin\cap\me=\varnothing$. Then
\begin{itemize}
\item[\sf (i)] \ $\vp: \bomin\to \ov{\vp(\omin)}\subset\h$  is finite and $\dim\omin=\dim\vp(\omin)$;
\item[\sf (ii)] \ the Lie algebra $\h$ is semisimple;
\item[\sf (iii)] \ Furthermore, if $H$ is spherical, then $H$ has a dense orbit in $\bomin$,
$\ov{\vp(\omin)}$, and $\ov{\psi(\omin)}$.
\end{itemize}
\end{lm}
\begin{proof}
{\sf (i)}  \ The proof is the same as in Proposition~\ref{prop:finite-proj}.
\\ \indent
{\sf (ii)}  Assume that the centre of $\h$, $\z(\h)$, is nontrivial. Then $\h$ is contained in the Levi 
subalgebra  $\el=\g^{\z(\h)}$. Here 
\[
   \hat\me=\el^\perp \subset \h^\perp=\me
\]
and $\hat\me$ contains a long root space {w.r.t.} a Cartan subalgebra $\te\subset\el$. Hence $\omin\cap\hat\me\ne\varnothing$.
\\ \indent
{\sf (iii)}  By Proposition~\ref{prop:codim-1}, $\max_{v\in\omin}\dim H{\cdot}v\ge \dim\omin-1$.
Since $\dim\omin$ is even, $\vp$ is finite, and the dimension of any $H$-orbit in $\h$ is even,
$H$ must have a dense orbit in $\ov{\vp(\omin)}$. This readily implies the rest.
\end{proof}

{\it\bfseries Remarks.} 1$^o$. It is {\bf not} claimed in part {\sf (iii)} that $\dim\omin=\dim\psi(\omin)$.
\\
2$^o$. The sphericity of $H$ is not necessary for $H$ to have a dense orbit in $\omin$, see below.

We already proved that, for $\h=\g_0$ and $\me=\g_1$, there are six cases in which 
$\omin\cap\me=\varnothing$ (Example~\ref{ex:6-cases}). However, it is not 
hard to provide examples with such property that are not related to involutions.

\begin{ex}    \label{ex:new-empty}
Let $\g$ be a simple Lie algebra with two root lengths and $\h=\g_{\sf lg}$ the subalgebra that is the sum
of $\te$ and the long root spaces. It is easily seen that $\h$ is semisimple. 
Here $\me$ is the sum of all short root spaces. Therefore, 
$\omin\cap\me=\varnothing$. This yields the following suitable pairs $(\g,\g_{\sf lg})$: 
\\[.7ex]
\centerline{
({\sf 1}$^o$)  $(\GR{B}{n},\GR{D}{n})$; \quad 
({\sf 2}$^o$) $(\GR{C}{n},\underbrace{\GR{C}{1}\dotplus\dots\dotplus\GR{C}{1}}_n)$;
\quad ({\sf 3}$^o$) $(\GR{F}{4},\GR{D}{4})$; 
\quad ({\sf 4}$^o$) $(\GR{G}{2},\GR{A}{2})$.}
\\[.6ex]
Pair (1$^o$) is symmetric, but the other three items provide new instances. Moreover, if
$\g_{\sf lg}$ is not a maximal subalgebra and $\g_{\sf lg}\subset\h'\subset\g$, then $\h'$ is also semisimple 
and $\omin\cap\me'=\varnothing$. Specifically, one can take any
$\h'=\GR{C}{n_1}\dotplus\dots\dotplus\GR{C}{n_k}$ with $n_1+\dots +n_k=n$ in (2$^o$) and
$\h'=\GR{B}{4}$ in (3$^o$). Here $H$ is not spherical in (2$^o$) for $n\ge 3$ and in (3$^o$). 
Furthermore, the 
above subalgebra $\h'$ for (2$^o$) is not spherical, if $k\ge 3$. Nevertheless, $H=G_{\sf lg}$ always has a dense orbit in $\omin$, regardless of sphericity.
\end{ex}

\subsection{Projections of spherical $G$-orbits in $\g$}    
\label{subs:proj-sph-orb}
Let $\g=\g_0\oplus\g_1$ be a $\BZ_2$-grading corresponding to an involution
of maximal rank $\vth_{\sf max}$. In this case $\g_1=\g_1^{(\vth_{\sf max})}$ contains a Cartan subalgebra
of $\g$ and  $\dim\g_1=\dim\be$. Recall that  $p_1:\g\to\g_1$ is the projection with kernel $\g_0$.
Let $\co\subset\g$ be an arbitrary $G$-orbit, not necessarily nilpotent.  
Write $\psi$ for the restriction of  $p_1$ to $\ov{\co}$. Whenever we wish to keep track of $\co$
in the projection, we write $\psi=\psimo$.

\begin{thm}             \label{thm:sph-proj}
If $\co$ is spherical, then $\dim\co=\dim\psi(\co)$, i.e., generic fibres of $\psi=\psimo$ are finite. 
\end{thm}
\begin{proof}
By definition, $\co$ is spherical if and only if there is $x\in \co$ such that $B{\cdot}x$ is dense in $\co$, 
i.e., $\g^x+\be=\g$. By~\cite{pavi}, the sphericity of $\co$ is equivalent to that
there is $z\in\co$ such that $\g^z+\g_1=\g$. Clearly, $\Omega=\{z\in \co\mid \g^z+\g_1=\g\}$ is open in
$\co$ (and dense, if $\co$ is spherical).  Taking the orthogonal complements with respect to $\Phi$ 
shows that $[\g,z]\cap\g_0=\{0\}$ for any $z\in\Omega$. Recall that $\mathbb T_z(Z)$ denotes the 
tangent space at $z\in Z$. Then
\[
   \mathbb T_z(\co)\cap\g_0=\{0\} .
\]
Since $\psi$ comes from the projection with kernel $\g_0$, the last equality means 
that the differential of $\psi$ at $z$ has trivial kernel, i.e.,
\[
    \textsl{d}_z\psi: \mathbb T_z(\co)\to \mathbb T_{\psi(z)}(\psi(\co)) 
\]
is bijective for generic $z\in \co$.
\end{proof}

For $\vth_{\sf max}$, one has $\co\cap\g_1\ne \varnothing$ for every $G$-orbit $\co\subset\g$ (Remark~\ref{rem:Levon}).
In particular, for $\co=\omin$, it follows from Theorem~\ref{thm:deg-psi=2} that $\deg(\psi_{\omin})=2$.

{\it \bfseries Question}. Is there a reasonable formula for $\deg(\psimo)$ for any spherical $G$-orbit 
$\co\subset\g$?

Let $\Nsp$ denote the union of all spherical nilpotent orbits.  It is a closed {\sl irreducible\/} subvariety of
$\N$~\cite[Sect.\,6]{aif99}. The irreducibility stems from the fact that there is a unique maximal spherical 
nilpotent orbit, which is denoted $\co^{\sf sph}$. Hence $\Nsp=\ov{\co^{\sf sph}}$. Furthermore, 
$\dim\Nsp=\dim\be-\ind\be$, where $\ind(\cdot)$ is the {\it index\/} of a Lie algebra, see~\cite[6.4]{aif99}.
That is, in our case $\dim\Nsp=\dim\g_1-\ind\be$. It is known that $\ind\be=0$ if and only if $-1\in W$, 
i.e., $\g\ne\GR{A}{n}, \GR{D}{2n+1}, \GR{E}{6}$. In this case, the projection
$\psim: \Nsp \to \g_1$ is dominant.  

In Proposition~\ref{prop:finite-proj} and Lemma~\ref{lm:secant-sigma}, we use the obvious fact that
$\sigma(\omin)=\omin$ for any $\sigma\in\mathsf{Inv}(\g)$. However, for $\sigma=\vth_{\sf max}$, the 
similar property holds for any $\co\subset\N$. Therefore, an analogue of Lemma~\ref{lm:secant-sigma} 
holds for any nilpotent orbit whenever $\sigma=\vth_{\sf max}$.


\begin{thebibliography}{P95}

\bibitem{leva} {\rusc L.V.~Antonyan}. 
{\rus O klassifikacii odnorodnykh {e1}lementov ${\BZ}_2$-graduirovannykh poluprostykh algebr Li},
{\rusi  Vestnik Mosk. Un-ta, Ser. Matem. Meh.}, {\rus N0}\,2 (1982), 29--34 (Russian). 
English translation: {\sc L.V.~Antonyan}. On classification of
homogeneous elements of ${\BZ}_2$-graded semisimple Lie algebras,
{\it  Moscow Univ. Math. Bulletin}, {\bf 37}\,(1982), {\rus N0}\,2, 36--43.

\bibitem{bk95} {\sc R.K.\,Brylinski} and {\sc B.~Kostant}.
Lagrangian models of minimal representations of $E_6$, $E_7$ and $E_8$. 
{\it Functional analysis on the eve of the 21st century}, Vol.~1, 13--63, 
Progr. Math., {\bf 131}, Birkh\"auser Boston, 1995.

\bibitem{bul}  {\sc M.~Bulois}. 
Sheets of symmetric Lie algebras and Slodowy slices,  
{\it J. Lie Theory}, {\bf 21}\,(2011), 1--54. 

\bibitem{chap}
{\sc P.-E.~Chaput}. Severi varieties, {\it Math. Z.}, {\bf 240}\,(2002), 451--459.

\bibitem{CM} {\sc D.H.~Collingwood} and {\sc W.~McGovern}.
{\it ``Nilpotent orbits in semisimple Lie algebras''},  New York: Van Nostrand Reinhold, 1993.

\bibitem{CP83}
{\sc C.~De Concini} and {\sc C.~Procesi}. Complete symmetric varieties, {\it Invariant theory} (Montecatini, 1982), 1--44, Lecture Notes in Math., {\bf 996}, Springer, Berlin, 1983.

\bibitem{ag72} 
{\rusc A.G.~{E1}lashvili}. {\rus Kanonicheski{\u\i} vid i statsionarnye podalgebry
tochek obwego polozheniya dlya prostykh line{\u\i}nykh grupp Li},
{\rusi Funkts.\ analiz i\  prilozh.} {\rus t}.{\bf 6}, {\rus N0}\,1 (1972), 51--62
(Russian). English translation: 
{\sc A.G.~Elashvili}. 
Canonical form and stationary subalgebras of points of
general position for simple linear Lie groups, {\it Funct. Anal. Appl.}
{\bf 6}\,(1972), 44--53.

\bibitem{he79}
{\sc W.~Hesselink}.  The normality of closures of orbits in a Lie algebra. 
{\it Comment. Math. Helv.} {\bf 54}\,(1979), no.\,1, 105--110.

\bibitem{kaji99} 
{\sc H.~Kaji}. Homogeneous projective varieties with degenerate secants,
{\it Trans. Amer. Math. Soc.}, {\bf 351}\,(1999), 533--545.

\bibitem{koy99} 
{\sc H.~Kaji, M.~Ohno} and {\sc O.~Yasukura}. 
Adjoint varieties and their secant varieties, 
{\it Indag. Math.} (N.S.) {\bf 10}\,(1999), no.~1, 45--57.

\bibitem{ky00} 
{\sc H.~Kaji} and {\sc O.~Yasukura}. 
Secant varieties of adjoint varieties: orbit decomposition, 
{\it J. Algebra},  {\bf 227}\,(2000), no.~1, 26--44.

\bibitem{kr71} {\sc B.~Kostant} and {\sc S.~Rallis}. 
Orbits and representations associated with symmetric spaces, 
{\it Amer. J. Math.}, {\bf 93}\,(1971), 753--809.

\bibitem{smt}
{\sc V.~Lakshmibai} and {\sc K.~Raghavan}.
``Standard monomial theory. Invariant theoretic approach'',  
Encyclopaedia Math. Sci., {\bf 137}. Springer-Verlag, Berlin, 2008. xiv+265 pp. 

\bibitem{land}
{\sc J.M.~Landsberg}.
``Tensors: geometry and applications''. Graduate Studies in Mathematics, {\bf 128}. AMS, Providence, RI, 2012. 

\bibitem{law95} {\sc R.~Lawther}. 
Jordan block sizes of unipotent elements in exceptional algebraic groups, 
{\it Comm. Alg.},  {\bf 23}\,(1995), 4125--4156.

\bibitem{L}
{\sc W.~Lichtenstein}. 
A system of quadrics describing the orbit of the highest weight vector, 
{\it Proc. Amer. Math. Soc.}, {\bf  84}\,(1982), no. 4, 605--608.

\bibitem{aif99} {\sc D.~Panyushev}.
On spherical nilpotent orbits and beyond,
{\it Ann. Inst. Fourier} {\bf 49}\,(1999), 1453--1476.

\bibitem{pavi} {\sc D.~Panyushev} and {\sc E.B.~Vinberg}.
An infinitesimal characterization of the complexity of homogeneous
spaces, {\it Transformation Groups}, {\bf 4}, no.\,1~(1999), 97--101. 

\bibitem{R67} {\sc R.W.~Richardson}.
Conjugacy classes in Lie algebras and algebraic groups, 
{\it Ann. of Math. (2)} {\bf 86}\,(1967), 1--15.

\bibitem{spr87} {\sc T.A.~Springer}. 
The classification of involutions of simple algebraic groups, 
{\it J. Fac. Sci. Univ. Tokyo} Sect. IA Math. {\bf 34}\,(1987), 655--670.

\bibitem{vi76}
{\rusc {E1}.B.~Vinberg}. {\rus Gruppa Ve\u\i lya 
graduirovanno{\u\i} algebry Li}, {\rusi Izv. AN SSSR. Ser. Matem.} 
{\bf 40}, {\rus N0}\,3 (1976), 488--526 (Russian). English translation: 
{\sc E.B.\,Vinberg}. The Weyl group of a graded Lie algebra,
{\it Math. USSR-Izv.} {\bf 10}\,(1976), 463--495.

\bibitem{t41}
{\rusc {E1}.B.~Vinberg, V.V.~Gorbatseviq, A.L.~Oniwik}.
``{\rusi Gruppy i algebry Li - 3}'', {\rus  
Sovrem. probl. matematiki. Fundam. napravl., t.\,{\bf 41}.
Moskva: VINITI} 1990 (Russian).
English translation: {\sc V.V.\,Gorbatsevich, A.L.\,Onishchik} and {\sc E.B.\,Vinberg}.
``{\it Lie Groups and Lie Algebras}" III
(Encyclopaedia Math. Sci., vol.~{\bf 41}) Berlin: Springer 1994.

\bibitem{vp72} {\rusc {E1}.B.~Vinberg, V.L.~Popov}.
{\rus Ob odnom klasse affinnyh kvaziodnorodnyh mnogoobrazi{\u\i}},
{\rusi Izv. AN SSSR. Ser. Matem.}, {\rus t.}{\bf 36}\,(1972),
749--764 (Russian). English translation:
{\sc E.B.~Vinberg} and {\sc V.L.~Popov}. On a class of quasihomogeneous
affine varieties, {\it Math. USSR-Izv.}, {\bf 6}\,(1972), 743--758.

\bibitem{vp89}
{\rusc {E1}.B.~Vinberg, V.L.~Popov}. {\rusi ``Teoriya Invariantov"}, {\rus   
Sovrem. probl. matematiki. Fundam. napravl., t.\,{\bf 55},
str.}\,137--309. {\rus Moskva: VINITI} 1989 (Russian).
English translation: 
{\sc V.L.~Popov} and {\sc E.B.~Vinberg}. ``Invariant theory", In: {\it  Algebraic Geometry IV}
(Encyclopaedia Math. Sci., vol.~{\bf 55}, pp.123--284) 
Berlin Heidelberg New York: Springer 1994.

\end{thebibliography}
\end{document}